\documentclass[amssymb, 11pt]{amsart}
\usepackage{latexsym, amssymb}

\newcommand{\cO}{\mathcal{O}}
\newcommand{\cG}{\mathcal{G}}
\newcommand{\cR}{\mathcal{R}}
\newcommand{\cN}{\mathcal{N}}
\newcommand{\cE}{\mathcal{E}}
\newcommand{\cM}{\mathcal{M}}
\newcommand{\cB}{\mathcal{B}}
\newcommand{\sa}{\mathsf{a}}
\newcommand{\Cx}{\mathbb{C}}

\newcommand{\Z}{\mathbb{Z}}
\newcommand{\F}{\mathbb{F}}
\newcommand{\sV}{\mathsf{T}}
\newcommand{\sW}{\mathsf{W}}
\newcommand{\sG}{\mathsf{G}}
\newcommand{\sR}{\mathsf{R}}
\newcommand{\sF}{\mathsf{F}}
\newcommand{\Irr}{\mathrm{Irr}}
\newcommand{\Ind}{\mathrm{Ind}}
\newcommand{\tF}{\tilde{F}}
\newcommand{\vep}{\varepsilon}
\newcommand{\epi}{\varepsilon_{\iota}}
\newcommand{\eps}{\varepsilon_{\sigma}}

\newcommand{\gO}{\mathrm{O}}
\newcommand{\SO}{\mathrm{SO}}
\newcommand{\CSO}{\mathrm{CSO}}
\newcommand{\PCSO}{\mathrm{PCSO}}
\newcommand{\CO}{\mathrm{CO}}
\newcommand{\Sp}{\mathrm{Sp}}
\newcommand{\CSp}{\mathrm{CSp}}
\newcommand{\GL}{\mathrm{GL}}
\newcommand{\GU}{\mathrm{U}}
\newcommand{\cP}{\mathcal{P}}
\newcommand{\cS}{\mathcal{S}}
\newcommand{\bG}{\mathbf{G}}
\newcommand{\bH}{\mathbf{H}}
\newcommand{\bT}{\mathbf{T}}

\newtheorem{theorem}{Theorem}[section]
\newtheorem{proposition}{Proposition}[section]
\newtheorem{lemma}{Lemma}[section]
\newtheorem{cor}{Corollary}[section]
\newtheorem*{thm}{Theorem}

\numberwithin{equation}{section}

\begin{document}

\title [Totally orthogonal finite simple groups]{Totally orthogonal finite simple groups}

\author{C. Ryan Vinroot}
\address{Department of Mathematics\\College of William and Mary\\ Williamsburg, VA, 23187\\USA}
\email{vinroot@math.wm.edu}


\begin{abstract}  We prove that if $G$ is a finite simple group, then all irreducible complex representations of $G$ may be realized over the real numbers if and only if every element of $G$ may be written as a product of two involutions in $G$.  This follows from our result that if $q$ is a power of 2, then all irreducible complex representations of the orthogonal groups $\gO^{\pm}(2n,\F_q)$ may be realized over the real numbers.  We also obtain generating functions for the sums of degrees of several sets of unipotent characters of finite orthogonal groups, and we obtain a twisted version of our main result for a broad family of finite classical groups.
\end{abstract}

\maketitle

\thispagestyle{empty}

\section{Introduction}

Suppose $G$ is a finite group and $\pi: G \rightarrow \GL(d, \Cx)$ is an irreducible complex representation of $G$.  Then $\pi$ may be realized over the real numbers if and only if there exists a non-degenerate $G$-invariant symmetric form $B$ on $V=\Cx^d$.  That is, a basis for $V= \Cx^d$ exists such that every resulting matrix for $\pi(g)$ has real entries precisely when $\pi$ maps $G$ into an orthogonal group $\gO_d(B)$ defined by $B$.  Thus such representations are called {\em orthogonal}, and if $G$ has the property that all of its irreducible complex representations are orthogonal then $G$ is called {\em totally orthogonal}.  While an irreducible representation of $G$ with real-valued character may or may not be orthogonal, it is known that the number of real-valued irreducible characters of $G$ is equal to the number of {\em real conjugacy classes} of $G$, which are conjugacy classes of elements which are conjugate to their own inverses.  Brauer asked \cite[Problem 14]{Br63} whether the number of irreducible orthogonal representations of a finite group $G$ can be expressed in terms of group-theoretical data.  A sub-question of this is whether we can give a group-theoretical condition which is equivalent to $G$ being totally orthogonal.  It is the latter question which we answer in this paper for the case that $G$ is a finite simple group. 

An element $g \in G$ is called {\em strongly real} if there exists some $h \in G$ such that $h^2 = 1$ and $hgh^{-1} = g^{-1}$, or equivalently if $g = h_1 h_2$ for some $h_1, h_2 \in G$ such that $h_1^2 = h_2^2 = 1$.  A finite group $G$ is {\em strongly real} if every element of $G$ is strongly real.  It was stated as a conjecture in \cite{KaKu}, attributed to P. H. Tiep, that a finite simple group $G$ is totally orthogonal if and only if it is strongly real.  It is the verification of this conjecture that is the main result of this paper:
\begin{thm} [Theorem \ref{SimpleTO}]  Let $G$ be a finite simple group.  Then $G$ is totally orthogonal if and only if $G$ is strongly real.
\end{thm}
We remark that this statement does not hold in general when $G$ is not simple, as given by the families of counterexamples found by Kaur and Kulshrestha in \cite{KaKu}.

A brief recent history of the progress on this result is as follows.  Tiep and Zalesski \cite{TZ05} classified precisely which finite simple groups are {\em real}, that is, have the property that all conjugacy classes are real.  A program then began of classifying which finite simple groups are strongly real, which was concluded in the papers \cite{VG10, Ra11}.  This program yielded the somewhat surprising result that all real finite simple groups are strongly real.  Other than several alternating groups and two sporadic groups, the rest of the groups on the list of strongly real finite simple groups are groups of Lie type.  When one considers whether these groups are totally orthogonal, as explained in \cite[Section 3]{TV17}, this follows for many of the groups from previous work, like the paper of Gow \cite{Gow85} which covers the simple symplectic and orthogonal groups in odd characteristic.  The only cases which do not follow from previous results are the symplectic groups $\Sp(2n, \F_q)$ when $q$ is a power of $2$, which this author proved in \cite{V17}, and the special orthogonal groups $\Omega^{\pm}(4m, \F_q) = \SO^{\pm}(4m, \F_q)$ when $q$ is a power of 2, which we complete here in Theorem \ref{SO4misTO}.  This statement follows from our result in Theorem \ref{OisTO} that the full orthogonal group $\gO^{\pm}(2n, \F_q)$ is totally orthogonal when $q$ is a power of $2$.  Since every real finite simple group is strongly real, and so totally orthogonal from our main result, then a finite simple group is real if and only if it is totally orthogonal.  This translates into a curious behavior of the Frobenius-Schur indicators of finite simple groups, which we give in Corollary \ref{SimpleFSCor}.  We point out that all of these results are dependent on the classification of finite simple groups, and it would be very satisfying to understand these properties without the classification.

We now give an outline of this paper and the main argument.  We give background on the finite orthogonal and special orthogonal groups in Section \ref{OrthoDef}.  In Section \ref{IndicatorsPre}, we give background on the standard and twisted Frobenius-Schur indicator, including the crucial (twisted) involution formula which says that the character degree sum for a finite group is equal to the number of (twisted) involutions if and only if all (twisted) indicators are 1.  In Section \ref{InvolsSec} we give the generating functions for the number of involutions in the finite orthogonal groups due to Fulman, Guralnick, and Stanton \cite{FGS}, and for the number of involutions in the finite special orthogonal groups or in its other coset in the orthogonal groups, from \cite{TV17}.  These generating functions motivate our main method in the following way.  When $q$ is odd, the Frobenius-Schur indicators of $\gO^{\pm}(2n, \F_q)$, $\SO^{\pm}(4m, \F_q)$, and $\SO^{\pm}(4m+2, \F_q)$ (a twisted indicator in the last case) are known to be $1$.  We may thus equate the generating functions for the relevant sets of involutions with a generating function for the character degree sum of the associated group.  Moreover, such an equality holds for $q$ even if and only if the indicators of interest are all 1 for that case.  To take advantage of these two facts, our main goal for the bulk of the paper is to directly calculate a generating function for the character degree sums of these groups through the character theory of finite reductive groups.  The crux of the argument is that we are able to obtain generating functions for the character degree sums of $\gO^{\pm}(2n, \F_q)$ and $\SO^{\pm}(2n, \F_q)$, where only a factor involving unipotent character degrees for these groups is not expanded as a nice infinite product.  But, the degrees of unipotent characters are the same polynomials in $q$ whether $q$ is even or odd, and since we have another generating function for these character degree sums when $q$ is odd through the involution count, we are able to solve for the generating function for sums of unipotent character degrees.  Using this, we are able to conclude the character degree sums in the case $q$ is even match those for the desired involution counts, giving the main results.  This is the same method which is employed for the case of $\Sp(2n, \F_q)$ with $q$ even (and $\SO(2n+1, \F_q)$ with $q$ odd) in \cite{V17}, although in the cases at hand there are several obstructions of order 2 which make the calculation more complicated which we now explain.

First, we must deal with both the split and non-split orthogonal and special orthogonal groups.  We take care of this by dealing with them both simultaneously, as the count for the number of involutions in both of these groups takes a nice form, and the combinatorics of their unipotent characters taken together is well-behaved.  Another point is that we have the groups $\SO^{\pm}(4m, \F_q)$ are totally orthogonal, while the groups $\SO^{\pm}(4m+2, \F_q)$ are not.  This is resolved by considering the aforementioned twisted indicators, twisted by the order 2 graph automorphism in the latter case.  These cases are also unified when considering the full orthogonal group, which are all totally orthogonal when $q$ is odd, although this introduces an order 2 obstruction again since the groups $\gO^{\pm}(2n, \F_q)$ are the $\F_q$-points of a disconnected algebraic group with 2 components.  To obtain the character degrees of the finite orthogonal and special orthogonal groups, one must understand the structure of centralizers in these groups of the semisimple elements in the special orthogonal groups, which is done in Section \ref{CentsSemi}.  Because these groups have center of order 2 when $q$ is odd, some of these centralizers are the groups of finite points of disconnected groups with 2 components.  We must also understand the unipotent characters of these centralizers, which are discussed in Section \ref{Unipotent}.  The disconnectedness of some centralizers means we must induce some unipotent characters from an index 2 subgroup and understand the behavior.  Through the Jordan decomposition of characters, the topic of Section \ref{JordanD}, we can obtain all character degrees of connected finite reductive groups through the orders of centralizers of semisimple elements, and the degrees of their unipotent characters.  Despite the disconnectedness of $\gO^{\pm}(2n, \F_q)$, a Jordan decomposition of characters was obtained by Aubert, Michel, and Rouquier \cite{AuMiRo} from the equivariance of the Jordan decomposition map for $\SO^{\pm}(2n, \F_q)$ under the order 2 graph automorphism (see Proposition \ref{SigmaEqui}), and a formula for character degrees of $\gO^{\pm}(2n, \F_q)$ results in Proposition \ref{JordanO}.

After the calculations are made in the above steps, we are able to compute generating functions for the sums of (modified) character degrees of $\gO^{\pm}(2n, \F_q)$ and $\SO^{\pm}(2n, \F_q)$ in Section \ref{GenFuns}.  The main results are then finally obtained in Section \ref{MainRs}, including combinatorial identities for the sums of (modified) character degrees of several sets of unipotent characters of $\SO^{\pm}(2n, \F_q)$, and while $\SO^{\pm}(4m+2, \F_q)$ is not totally orthogonal, we do show that all real-valued characters are orthogonal in Theorem \ref{SOIndicators1}.  In Section \ref{symplectic}, we use one of the main results to give a significantly shorter proof of the fact that certain twisted indicators of $\Sp(2n, \F_q)$ are all 1 when $q$ is odd, which was first proved in \cite{V05}.  In the last Section \ref{Twizted}, we expand on the above and other results to obtain a sort of twisted version of the main result Theorem \ref{SimpleTO} for certain finite classical groups.  Namely, in Theorem \ref{TwistList} we show that there is an order 2 automorphism for all simple adjoint classical algebraic groups over finite fields, such that every element is a product of two twisted involutions, and every irreducible representation has twisted indicator $1$.  This suggests a larger picture which we hope to understand in the future.
\\
\\
\noindent{\bf Acknowledgements. } The author thanks Mandi Schaeffer Fry and Jay Taylor for helpful correspondence regarding Proposition \ref{SigmaEqui}.  The author was supported in part by a grant from the Simons Foundation, Award \#280496.

\section{Preliminaries}

\subsection{Orthogonal groups over finite fields}  \label{OrthoDef}
We follow the construction of the even-dimensional orthogonal groups over finite fields as given in \cite[Example 22.9]{MaTe11}.  Let $q$ be a power of a prime $p$, and let $\F_q$ be a finite field with $q$ elements with fixed algebraic closure $\overline{\F}_q$.  On the $\overline{\F}_q$-vector space $V=\overline{\F}_q^{2n}$, where $v \in V$ is given by coordinates $v=(x_1, \ldots x_{2n})$, define the quadratic form $Q$ by 
$$ Q(v) = x_1 x_{2n} + \cdots + x_n x_{n+1}.$$
The stabilizer of the form $Q$ in the general linear group $\GL_{2n} = \GL(2n, \overline{\F}_q)$ is the orthogonal group with respect to $Q$, denoted $\gO_{2n} = \gO_{2n}(Q) = \gO(2n, \overline{\F}_q)$.  This is a disconnected group, and the connected component $\gO_{2n}^{\circ}$ is defined to be the special orthogonal group with respect to $Q$, so $\gO_{2n}^{\circ} = \SO_{2n} = \SO(2n, \overline{\F}_q)$.

Now take $F$ to be the standard Frobenius endomorphism on $\GL_{2n}$ with respect to $\F_q$, so if $g = (a_{ij}) \in \GL_{2n}$, then $F(g) = (a_{ij}^q)$.  Restrict this map to $\gO_{2n}$, and the fixed points define the \emph{orthogonal group over $\F_q$ of $+$-type}, and the $F$-fixed points of $\SO_{2n}$ form the \emph{special orthogonal group over $\F_q$ of $+$-type}:
$$ \gO^+(2n, \F_q) = \gO_{2n}^F, \quad \SO^+(2n, \F_q) = \SO_{2n}^F.$$
These are also called the \emph{split} orthogonal and special orthogonal groups over $\F_q$.

Now take the element $h \in \gO_{2n}$ defined by
$$ h = \left( \begin{array} {llll} I_{n-1} &  &  &  \\  &  0 & 1 &  \\   &  1 & 0  &  \\  &  &  & I_{n-1} \end{array} \right).$$
Then $h \in \gO_{2n} \setminus \SO_{2n}$ (whether $q$ is even or odd).  Let $\sigma$ denote the automorphism on $\gO_{2n}$ and $\SO_{2n}$ defined by conjugation by $h$.  This defines a graph automorphism of order 2 on the Dynkin diagram of $\SO_{2n}$.  Now the map $\tF = \sigma F$ defines a twisted Frobenius morphism (or a Steinberg morphism) on the groups $\gO_{2n}$ and $\SO_{2n}$.  The \emph{orthogonal} and \emph{special orthogonal goups over $\F_q$ of $-$-type} are defined to be the groups of $\tF$-fixed points of $\gO_{2n}$ and $\SO_{2n}$:
$$ \gO^-(2n, \F_q) = \gO_{2n}^{\tF}, \quad \SO^-(2n, \F_q) = \SO_{2n}^{\tF},$$
also called the \emph{non-split} orthogonal and special orthogonal groups over $\F_q$.  We note that we could also define the split or non-split groups as the stabilizers of inequivalent quadratic forms on $\F_q^{2n}$.

When speaking of either the split or non-split orthogonal or special orthogonal groups, we will use the notation $\gO^{\pm}(2n, \F_q)$ or $\SO^{\pm}(2n, \F_q)$, respectively.  Note that we have
\begin{equation} \label{OSOindex2}
\gO^{\pm}(2n, \F_q) = \langle \SO^{\pm}(2n, \F_q), h \rangle \cong \SO^{\pm}(2n, \F_q) \rtimes \langle \sigma \rangle.
\end{equation}

In the case that $q$ is even, the groups $\SO^{\pm}(2n, \F_q)$ are finite simple groups, and are also denoted by $\Omega^{\pm}(2n, \F_q)$ (and are the derived groups of the finite orthogonal group in this case).  

The orders of the groups defined above are as follows, for $n \geq 1$:
$$|\gO^{\pm}(2n, \F_q)| = 2q^{n^2 - n} (q^n \mp 1) \prod_{i=1}^{n-1} (q^{2i} - 1) \text{ and } |\SO^{\pm}(2n, \F_q)| = \frac{1}{2} |\gO^{\pm}(2n, \F_q)|.$$
While we do not define these groups in the case $n=0$, we will several times need to define notions for the $n=0$ case for the purpose of constant terms in power series.

\subsection{Indicators}  \label{IndicatorsPre}

Let $G$ be a finite group with $\iota$ an automorphism of $G$ satisfying $\iota^2=1$.  Suppose $(\pi, V)$ is an irreducible complex representation of $G$ satisfying $\iota.\pi \cong \hat{\pi}$, where $(\hat{\pi}, \hat{V})$ is the dual representation of $(\pi, V)$, and the representation $(\iota.\pi, V)$ is defined by $\iota.\pi = \pi \circ \iota$.  It follows from Schur's Lemma that there exists a non-degenerate bilinear form $B$ on $V$, unique up to scalar multiple, which satisfies
\begin{equation} \label{BiFormInd}
B(\iota.\pi(g)v, \pi(g)w) = B(v, w),
\end{equation}
for all $g \in G$ and all $v, w \in V$.  By the uniqueness of $B$ and exchanging the variables, it follows that $B$ must be either a symmetric or an alternating form.  We define the \emph{twisted Frobenius-Schur indicator} of $\pi$ with respect to $\iota$, denoted $\epi(\pi)$, as the sign which appears when exchanging the variables of $B$.  That is,
\begin{equation} \label{BiFormInd2} 
B(v, w) = \epi(\pi) B(w, v)
\end{equation}
for all $v, w \in V$.  If $\iota.\pi \not\cong \hat{\pi}$, then we define $\epi(\pi) =0$. 

For the case that $\iota$ has order 2, the twisted Frobenius-Schur indicator was defined by Kawanaka and Matsuyama \cite{KaMa90}, and further investigated and generalized by Bump and Ginzburg \cite{BuGi04}, and all of its basic properties we give in this section can be found in those papers.  When $\iota=1$ is the trivial automorphism, then $\epi(\pi) = \vep(\pi)$ is the classical Frobenius-Schur indicator, and the properties in this case can be found in standard references, for example \cite[Chapter 4]{Isaacs}.

Let $\chi$ be the irreducible character of the representation $(\pi, V)$, and we let $\Irr(G)$ denote the set of all irreducible characters of $G$.  The condition $\iota.\pi \cong \hat{\pi}$ is equivalent to $\iota.\chi = \bar{\chi}$, where $\iota.\chi = \chi \circ \iota$.  We also write $\epi(\chi) = \epi(\pi)$, and refer to the twisted Frobenius-Schur indicator of the character $\chi$.  Given some $\Cx$-basis $\cB$ of $V$, let $[\pi]_{\cB}$ denote the resulting matrix representation corresponding to $(\pi, V)$.  The twisted Frobenius-Schur indicator also has the following characterization:
$$ \epi(\chi) = \left\{ \begin{array}{rl} 1 & \text{ if there is some } \cB \text{ such that } [\iota.\pi]_{\cB} = \overline{[\pi(g)]_{\cB}} , \\ -1 & \text{ if } \iota.\chi = \bar{\chi} \text{ but there is no } \cB \text{ such that } [\iota.\pi]_{\cB} = \overline{[\pi]_{\cB}}, \\ 0 & \text{ if } \iota.\chi \neq \bar{\chi}. \end{array} \right. $$
In the classical case when $\iota=1$, this says that $\vep(\chi) = \pm 1$ if and only if $\chi$ is a real-valued character, and $\vep(\chi) = 1$ exactly when $(\pi, V)$ may be realized as a real representation.  By \eqref{BiFormInd} and \eqref{BiFormInd2}, $\vep(\chi) = 1$ is equivalent to being able to embed the image of $\pi$ in an orthogonal group (defined by $B$).  As mentioned in the introduction, for this reason we say that the finite group $G$ is \emph{totally orthogonal} if $\vep(\chi) = 1$ for all $\chi \in \Irr(G)$.  The condition $\epi(\chi) = \pm 1$ for all $\chi \in \Irr(G)$ is equivalent to $\iota(g)$ being $G$-conjugate to $g^{-1}$ for all $g \in G$.  We say $G$ is a \emph{real group} if $g$ and $g^{-1}$ are $G$-conjugate for all $g \in G$, since this is equivalent to all $\chi \in \Irr(G)$ being real-valued.

Given $\chi \in \Irr(G)$, one may compute $\epi(\chi)$ with the following formula:
$$ \epi(\chi) = \frac{1}{|G|} \sum_{g \in G} \chi(g \, \iota(g)).$$
From this formula and orthogonality relations for characters, we obtain the following generalization of the classical Frobenius-Schur involution formula:
$$ \sum_{ \chi \in \Irr(G)} \epi(\chi) \chi(1) = \# \{ g \in G \, \mid \, \iota(g) = g^{-1} \}.$$
We may immediately conclude the following, which is crucial to our main arguments in this paper.

\begin{lemma} \label{CharSumLemma}
Let $G$ be a finite group, and $\iota$ an automorphism of $G$ satisfying $\iota^2 = 1$.  Then $\epi(\chi) = 1$ for all $\chi \in \Irr(G)$ if and only if
$$ \sum_{\chi \in \Irr(G)} \chi(1) = \# \{ g \in G \, \mid \, \iota(g) = g^{-1} \}.$$
In particular, $G$ is totally orthogonal if and only if
$$ \sum_{\chi \in \Irr(G)} \chi(1) = \# \{ g \in G \, \mid \, g^2 = 1 \}.$$
\end{lemma} 

Consider the specific case of a finite group $G$ with $[G:H]=2$, such that $G$ is a split extension of $H$.  That is, we have $G = \langle H, y \rangle$ for some $y \in G \setminus H$ such that $y^2 = 1$.  We may then define the order 2 automorphism $\iota$ on $H$ by $\iota(h) = y h y^{-1}$, and we may consider the indicators $\epi(\chi)$ for $\chi \in \Irr(H)$.  The following result is given in \cite[Lemma 2.3]{TV17}, and will be relevant to our main groups of interest, as in \eqref{OSOindex2}.

\begin{lemma}  \label{index2} Let $G$ be a finite group with $H \leq G$ such that $[G:H]=2$, with $G = H \cup yH = \langle H, y \rangle$ and $y^2 = 1$.  Define $\iota$ on $H$ by $\iota(h) = y h y^{-1}$.  Then the following hold.
\begin{itemize}
\item[(i)] If $G$ is totally orthogonal and $H$ is a real group, then $H$ is totally orthogonal.
\item[(ii)] If $G$ is totally orthogonal, and $\iota(h)$ is $H$-conjugate to $h^{-1}$ for all $h \in H$, then for all $\chi \in \Irr(H)$ we have $\epi(\chi) = 1$ and $\vep(\chi) \geq 0$.
\end{itemize}
\end{lemma}

The following observation, which is \cite[Lemma 2.4]{TV17}, will be applied in the last section of this paper.

\begin{lemma} \label{modZ} Let $G$ be a finite group with automorphism $\iota$ such that $\iota^2 = 1$, and let $Z$ be the center of $G$.  Let $\iota$ also denote the automorphism of $G/Z$ defined by $\iota(gZ) = \iota(g)Z$.  If $\epi(\chi) = 1$ for all $\chi \in \Irr(G)$, then $\epi(\rho) = 1$ for all $\rho \in \Irr(G/Z)$.
\end{lemma}

\section{Involutions in orthogonal groups} \label{InvolsSec}

We define an \emph{involution} in a group $G$ to be any element $g \in G$ such that $g^2 = 1$, where we include the identity element as an involution for convenience.  For any subset $X$ of a finite group $G$, we let $I(X)$ denote the number of involutions in $X$.  For any finite group $G$, we let $\Sigma(G)$ denote the sum of the degrees of the irreducible characters of $G$.

From results of Fulman, Guralnick, and Stanton \cite[Theorem 2.11, Theorem 2.12, and Lemma 5.1]{FGS} the number of involutions in $\gO^{\pm}(2n, \F_q)$ for $q$ odd may be obtained by the generating function
\begin{align*}
a_{0}^\pm & + \sum_{n \geq 1} \frac{I(\gO^{\pm}(2n, \F_q))}{|\gO^{\pm}(2n, \F_q)|} q^{n^2} z^n \\ & = \frac{1}{2(1-zq)} \prod_{i \geq 1} \frac{(1 + z/q^{2(i-1)})^2}{1 - z^2/q^{2(i-1)}} \pm \frac{1}{2} \prod_{i \geq 1} \frac{(1 + z/q^{2i-1})^2}{1 - z^2/q^{2(i-1)}},
\end{align*}
where $a_{0}^{+} = 1$ and $a_{0}^{-} = 0$.  It then follows, by replacing $z$ with $z/q$ and taking the sum of the two resulting generating functions for the split and non-split orthogonal groups, that when $q$ is odd,
$$\frac{I(\gO^{+}(2n, \F_q))}{2(q^n -1) \prod_{i=1}^{n-1} (q^{2i} - 1) }  + \frac{I(\gO^{-}(2n, \F_q))} {2(q^n + 1) \prod_{i=1}^{n-1} (q^{2i} - 1)}$$ 
is the coefficient of $z^n$, $n > 0$, in the generating function
$$ \frac{1}{1 - z} \prod_{i \geq 1} \frac{(1 + z/q^{2i-1})^2}{1 - z^2/q^{2i}}.$$
Very similarly, when $q$ is even it follows from \cite[Theorems 2.14, 2.15, and 5.6]{FGS} that
\begin{align*}
a_0^{\pm}  & + \sum_{n \geq 1} \frac{I(\gO^{\pm}(2n, \F_q))}{|\gO^{\pm}(2n, \F_q)|} q^{n^2} z^n \\ & = \frac{1}{2(1-zq)} \prod_{i \geq 1} \frac{1 + z/q^{2(i-1)}}{1 - z^2/q^{2(i-1)}} \pm \frac{1}{2} \prod_{i \geq 1} \frac{1 + z/q^{2i-1}}{1 - z^2/q^{2(i-1)}},
\end{align*}
with $a_0^{\pm}$ as above.  When replacing $z$ with $z/q$ and adding the two generating functions, we have that when $q$ is even, 
$$\frac{I(\gO^{+}(2n, \F_q))}{2(q^n -1) \prod_{i=1}^{n-1} (q^{2i} - 1)}   +  \frac{I(\gO^{-}(2n, \F_q))}{2 (q^n + 1) \prod_{i=1}^{n-1} (q^{2i} - 1)}$$ 
is the coefficient of $z^n$, $n > 0$, in the generating function
$$ \frac{1}{1 - z} \prod_{i \geq 1} \frac{1 + z/q^{2i-1}}{1 - z^2/q^{2i}}.$$

It is a result of Gow \cite[Theorem 1]{Gow85} that when $q$ is odd, the groups $\gO^{\pm}(2n, \F_q)$ are totally orthogonal.  From this result, the generating functions just given, and Lemma \ref{CharSumLemma}, we obtain the following result, which motivates the methods in this paper.

\begin{proposition} \label{IndicatorsO}  The following hold:
\begin{itemize}
\item[(i)]  For any odd $q$, the coefficient of $z^n$, $n>0$, in the generating function
$$ \frac{1}{1 - z} \prod_{i \geq 1} \frac{(1 + z/q^{2i-1})^2}{1 - z^2/q^{2i}}$$
is given by
$$\frac{\Sigma(\gO^{+}(2n, \F_q))}{2(q^n -1) \prod_{i=1}^{n-1} (q^{2i} - 1) }  + \frac{\Sigma(\gO^{-}(2n, \F_q))} {2(q^n + 1) \prod_{i=1}^{n-1} (q^{2i} - 1)}.$$ 
\item[(ii)] For $q$ even, the groups $\gO^{\pm}(2n, \F_q)$ are totally orthogonal for all $n>0$ if and only if the coefficient of $z^n$, $n >0$, in the generating function
$$ \frac{1}{1 - z} \prod_{i \geq 1} \frac{1 + z/q^{2i-1}}{1 - z^2/q^{2i}}$$
is given by
$$\frac{\Sigma(\gO^{+}(2n, \F_q))}{2(q^n -1) \prod_{i=1}^{n-1} (q^{2i} - 1) }  + \frac{\Sigma(\gO^{-}(2n, \F_q))} {2(q^n + 1) \prod_{i=1}^{n-1} (q^{2i} - 1)}.$$ 
\end{itemize}
\end{proposition}

It will be convenient to define $e = e(q)$ to be the following for the remainder of the paper:
\begin{equation} \label{edefn}
e = e(q) = \left\{ \begin{array}{ll} 1 & \text{if } q \text{ is even,} \\ 2 & \text{if } q \text{ is odd.} \end{array} \right.
\end{equation}

We now consider involutions in the special orthogonal groups.  By applying \cite[Theorems 6.1(1) and 6.5]{TV17}, substituting $z$ with $z/q$ and multiplying by $2$, we have
\begin{align} 
b_0^{\pm} + \sum_{n \geq 1} \frac{I(\SO^{\pm}(2n, \F_q))}{|\SO^{\pm}(2n, \F_q)|} & q^{n^2 - n} z^n  = \prod_{i \geq 1} \frac{(1 + z/q^{2i-1})^e}{1 - z^2/q^{2i-2}} \pm \prod_{i \geq 1} \frac{(1 + z/q^{2i})^e}{1 - z^2/q^{2i}}  \nonumber\\
& = \frac{1}{1 - z^2} \prod_{i \geq 1} \frac{(1 + z/q^{2i-1})^e}{1 - z^2/q^{2i}} \pm \prod_{i \geq 1} \frac{(1 + z/q^{2i})^e}{1 - z^2/q^{2i}}, \label{InvSO}
\end{align}
where $b_0^+ = 2$ and $b_0^- = 0$.  By \cite[Theorems 6.2 and 6.6]{TV17}, and again substituting $z$ with $z/q$ and multiplying by $2$, we obtain
\begin{align}
\sum_{n \geq 1} \frac{I(\gO^{\pm}(2n, \F_q) \setminus \SO^{\pm}(2n, \F_q))}{|\SO^{\pm}(2n, \F_q)|}  q^{n^2 - n} z^n & = z \prod_{i \geq 1} \frac{(1 + z/q^{2i-1})^e}{ 1 - z^2 /q^{2i-2}}  \nonumber\\
& = \frac{z}{1 - z^2} \prod_{i \geq 1} \frac{(1 + z/q^{2i-1})^e}{ 1 - z^2 /q^{2i}}. \label{InvO-SO}
\end{align}
If we take only the even terms of \eqref{InvSO} and add this to only the odd terms of \eqref{InvO-SO}, we obtain 
\begin{align*}
\frac{1}{2} &\left(\frac{1}{1 - z^2} \prod_{i \geq 1} \frac{(1 + z/q^{2i-1})^e}{1 - z^2/q^{2i}} \pm \prod_{i \geq 1} \frac{(1 + z/q^{2i})^e}{1 - z^2/q^{2i}}  \right. \\
&  + \frac{1}{1 - z^2} \prod_{i \geq 1} \frac{(1 - z/q^{2i-1})^e}{1 - z^2/q^{2i}} \pm \prod_{i \geq 1} \frac{(1 - z/q^{2i})^e}{1 - z^2/q^{2i}} \\
& \left. + \frac{z}{1 - z^2} \prod_{i \geq 1} \frac{(1 + z/q^{2i-1})^e}{ 1 - z^2 /q^{2i}} + \frac{z}{1 - z^2} \prod_{i \geq 1} \frac{(1 - z/q^{2i-1})^e}{ 1 - z^2 /q^{2i}} \right).
\end{align*}
Taking the sum of the above for $+$-type and $-$-type groups, we obtain the generating function
\begin{equation} \label{SOgoal}
\frac{1}{1-z} \prod_{i \geq 1} \frac{(1 + z/q^{2i-1})^e}{1 - z^2/q^{2i}} + \frac{1}{1-z} \prod_{i \geq 1} \frac{(1 - z/q^{2i-1})^e}{1 - z^2/q^{2i}}.
\end{equation}
That is, if we define for $n>0$
$$ J(\SO^{\pm}(2n, \F_q)) = \left\{ \begin{array}{ll} I(\SO^{\pm}(2n, \F_q) & \text{if } n \text{ is even,} \\ I(\gO^{\pm}(2n, \F_q) \setminus \SO^{\pm}(2n, \F_q)) & \text{if } n \text{ is odd,} \end{array} \right.$$
then the generating function \eqref{SOgoal} has as its coefficient of $z^n$, $n>0$, the expression
$$\frac{J(\SO^{+}(2n, \F_q))}{(q^n -1) \prod_{i=1}^{n-1} (q^{2i} - 1)}   +  \frac{J(\SO^{-}(2n, \F_q))}{ (q^n + 1) \prod_{i=1}^{n-1} (q^{2i} - 1)}.$$ 
It was proved by Gow \cite[Theorem 2]{Gow85} that when $q$ is odd, $\SO^{\pm}(4m, \F_q)$ is totally orthogonal.  It was proved by G. Taylor and this author \cite[Theorem 5.1(ii)]{TV17} that when $q$ is odd, $\eps(\chi) = 1$ for all complex irreducible characters $\chi$ of $\SO^{\pm}(4m+2, \F_q)$, where $\sigma$ is the order 2 automorphism as in Section \ref{OrthoDef}, and it follows that the character degree sum for $\SO^{\pm}(4m+2, \F_q)$ with $q$ odd is exactly $I(\gO^{\pm}(4m+2, \F_q) \setminus \SO^{\pm}(4m+2, \F_q))$.  

From the discussion in the above paragraphs, together with Lemma \ref{CharSumLemma}, we have the next result.

\begin{proposition} \label{IndicatorsSO}  The following hold:
\begin{itemize}
\item[(i)]  For $q$ odd, the coefficient of $z^n$, $n >0$, in the generating function
$$\frac{1}{1-z} \prod_{i \geq 1} \frac{(1 + z/q^{2i-1})^2}{1 - z^2/q^{2i}} + \frac{1}{1-z} \prod_{i \geq 1} \frac{(1 - z/q^{2i-1})^2}{1 - z^2/q^{2i}} $$
is given by
$$\frac{\Sigma(\SO^{+}(2n, \F_q))}{(q^n -1) \prod_{i=1}^{n-1} (q^{2i} - 1) }  + \frac{\Sigma(\SO^{-}(2n, \F_q))} {(q^n + 1) \prod_{i=1}^{n-1} (q^{2i} - 1)}.$$ 
\item[(ii)] For $q$ even, the groups $\SO^{\pm}(4m, \F_q)$ are totally orthogonal, and every irreducible character $\chi$ of $\SO^{\pm}(4m+2, \F_q)$ satisfies $\eps(\chi) = 1$, if and only if the coefficient of $z^n$, $n>0$, in the generating function
$$ \frac{1}{1-z} \prod_{i \geq 1} \frac{1 + z/q^{2i-1}}{1 - z^2/q^{2i}} + \frac{1}{1-z} \prod_{i \geq 1} \frac{1 - z/q^{2i-1}}{1 - z^2/q^{2i}} $$
is given by
$$\frac{\Sigma(\SO^{+}(2n, \F_q))}{(q^n -1) \prod_{i=1}^{n-1} (q^{2i} - 1) }  + \frac{\Sigma(\SO^{-}(2n, \F_q))} {(q^n + 1) \prod_{i=1}^{n-1} (q^{2i} - 1)}.$$ 
\end{itemize}
\end{proposition}

Given Propositions \ref{IndicatorsO} and \ref{IndicatorsSO}, our main goal through the next several sections will be to compute the generating functions involving $\Sigma(\gO^{\pm}(2n, \F_q))$ and $\Sigma(\SO^{\pm}(2n, \F_q))$ directly, using the character theory of these groups.

\section{Centralizers of semisimple elements} \label{CentsSemi}

The conjugacy classes of the orthogonal groups $\gO^{\pm}(2n, \F_q)$, and the corresponding centralizers, were described by G. E. Wall.  The case that $q$ is odd is covered in \cite[Sec. 2.6 Case (C)]{Wall}, and the $q$ even case is covered in \cite[Sec. 3.7]{Wall}.  One may obtain the description of the semisimple classes which are contained in $\SO^{\pm}(2n, \F_q)$ and their centralizers in both $\gO^{\pm}(2n, \F_q)$ and $\SO^{\pm}(2n, \F_q)$ from this resource.  For the case that $q$ is odd, much of this information can also be found in \cite[Chapter 16]{CaEn04}, and some details for the $q$ even case can be found in \cite[Sec. 2.4]{FST}.

In order to describe the semsimple conjugacy classes of $\SO^{\pm}(2n, \F_q)$ or $\gO^{\pm}(2n, \F_q)$, we recall the notion of self-dual polynomials.  If $f(t) \in \F_q[t]$ is a monic polynomial of degree $d$ with nonzero constant term, $f(t) = t^d + a_{d-1} t^{d-1} + \cdots + a_1 t + a_0$, then the \emph{dual} polynomial of $f(t)$, which we denote by $f^*(t)$, is
$$ f^*(t) = a_0^{-1} t^d f(t^{-1}),$$
and $f(t)$ is \emph{self-dual} if $f(t) = f^*(t)$.  Then $f(t)$ is self-dual if and only if, for any root $\alpha \in \overline{\F}_q^{\times}$ of $f(t)$ with multiplicity $m$, $\alpha^{-1}$ is also a root of $f(t)$ with multiplicity $m$.  The elementary divisors of any element of $\gO^{\pm}(2n, \F_q)$ are all self-dual.  We let $\cN(q)$ denote the set of monic irreducible self-dual polynomials with nonzero constant in $\F_q[t]$, and we let $\cN(q)' = \cN(q) \setminus \{ t+1, t-1 \}$.  If $f(t) \in \cN(q)'$, then $\deg(f)$ is even, by \cite[Lemma 1.3.16]{FNP}.  Denote by $\cM(q)$ the set of unordered pairs $\{g, g^*\}$ of monic irreducible polynomials with nonzero constant in $\F_q[t]$ such that $g(t) \neq g^*(t)$.  

In the following description of semisimple classes, there is the sign $\pm$ corresponding to the split or non-split orthogonal or special orthogonal group, and there will often be a sign associated with $1$ or $-1$ eigenvalues.  For the rest of the paper, whenever there is a $\pm$ or $\mp$ in an expression, it will denote the difference between the cases of the split and non-split groups.  When there is a term which could be associated with elementary divisors which are powers of $t+1$ or $t-1$, we will denote the sign possibility as $+/-$.

We first describe semisimple classes $(s)$ of $G=\gO^{\pm}(2n, \F_q)$ which are contained in $H=\SO^{\pm}(2n, \F_q)$.  First assume that $q$ is odd, so that $s$ has $-1$ as an eigenvalue with even multiplicity.  Then the semisimple class $(s)$ is determined by the elementary divisors of $s$ unless $s$ has both $1$ and $-1$ as eigenvalues, in which cases there are two classes with the same elementary divisors.  In particular, the elementary divisors of $s$ must be of the form $f(t)^{m_f}$ with $f(t) \in \cN(q)'$, or $g(t)^{m_g}$ and $g^*(t)^{m_g}$ for $\{ g(t), g^*(t) \} \in \cM(q)$, or $(t+1)^{2m_+}$ or $(t-1)^{2m_-}$.  By adapting \cite[Proposition 16.8]{CaEn04}, we have that the semisimple class $(s)$ of $\gO^{\pm}(2n, \F_q)$ depends on a pair of functions 
\begin{equation} \label{SemiFunctions}
\Phi: \cN(q) \cup \cM(q) \rightarrow \Z_{\geq 0},  \quad \eta: \{ t -1, t+1 \} \rightarrow \{ +, - \}, 
\end{equation}
where we write $\Phi(f) = m_f$ for $f \in \cN(q)'$, $\Phi(\{ g, g^* \}) = m_g$ for $\{ g, g^* \} \in \cM(q)$, $\Phi(t + 1) = m_{+}$, $\eta(t+1) = \eta(+)$, $\Phi(t-1) =m_-$, and $\eta(t-1) = \eta(-)$, such that we have
\begin{equation} \label{SemiCond1}
|\Phi| := \sum_{f \in \cN'(q)} m_f \deg(f)/2 + \sum_{ \{g, g^* \} \in \cM(q) } m_g \deg(g) + m_+ + m_- = n,
\end{equation}
where
\begin{equation} \label{SemiCond2}
\eta(+/-) = + \, \text{ if } \, m_{+/-} = 0, 
\end{equation}
and such that
\begin{equation} \label{SemiCond3}
\eta(+) \eta(-) 1 \prod_{f \in \cN'(q)} (-1)^{m_f} = \tau 1,
\end{equation}
where $\tau$ is the sign $\pm$ associated to the reference orthogonal group $G = \gO^{\pm}(2n, \F_q) = \gO^{\tau}(2n, \F_q)$.  Given the semisimple class $(s)$ determined by the pair of functions $(\Phi, \eta)$, the centralizer of $s$ in $G = \gO^{\pm}(2n, \F_q)$ has isomorphism type given by
\begin{align}
C_G(s) \cong \prod_{f \in \cN(q)'} \GU(m_f, \F&_{q^{\deg(f)/2}}) \times \prod_{ \{ g, g^*\}  \in \cM(q)} \GL(m_g, \F_{q^{\deg(g)}}) \label{CentO}\\
& \times \gO^{\eta(+)} (2m_+, \F_q) \times \gO^{\eta(-)} (2m_-, \F_q), \nonumber
\end{align}
where $\GU(n, \F_q)$ denotes the full unitary group defined over $\F_q$ (or $n$-by-$n$ unitary matrices with entries from $\F_{q^2}$).

When $q$ is even, the description of semisimple classes $(s)$ of $\gO^{\pm}(2n, \F_q)$ (all of which are contained in $\SO^{\pm}(2n, \F_q)$) is very similar to the case that $q$ is odd, except $1$ and $-1$ eigenvalues are the same.  In particular, such a semisimple class is determined by a pair of functions $(\Phi, \eta)$ which satisfies \eqref{SemiFunctions} if we omit $\Phi(t-1)$, \eqref{SemiCond1} if we omit $m_-$, \eqref{SemiCond2} if we replace $\eta(+/-)$ with $\eta(+)$ and $m_{+/-}$ with $m_+$, \eqref{SemiCond3} if we omit $\eta(-)$, and with centralizer structure as in \eqref{CentO} if we omit the factor $\gO^{\eta(-)}(2m_-, \F_q)$.  When discussing the case of general $q$, we will always assume these omissions in the case that $q$ is even.

Note that in the general case, if we choose a pair of functions $(\Phi, \eta)$ which satisfies the conditions \eqref{SemiFunctions}, \eqref{SemiCond1}, and \eqref{SemiCond2}, then this specifies a unique semisimple class of exactly one type of orthogonal group $\gO^{\tau}(2n, \F_q)$, where the sign $\tau \in \{+, - \}$ is determined by the condition \eqref{SemiCond3}.

Now consider the centralizer of the semisimple element $s$ in the special orthogonal group $H = \SO^{\pm}(2n, \F_q)$.  Supposing that $q$ is odd and $s$ has both $1$ and $-1$ eigenvalues, then the structure of the centralizer $C_H(s)$ is given by (see \cite[Sec. 1B]{AuMiRo}) as
\begin{align}
C_H(s) \cong & \prod_{f \in \cN(q)'} \GU(m_f, \F_{q^{\deg(f)/2}}) \times \prod_{ \{ g, g^*\}  \in \cM(q)} \GL(m_g, \F_{q^{\deg(g)}}) \label{CentSO}\\
& \times (\SO^{\eta(+)} (2m_+, \F_q) \times \SO^{\eta(-)} (2m_-, \F_q)) \rtimes \langle \sigma_s \rangle, \nonumber
\end{align}
where $\sigma_s$ acts as the outer automorphism $\sigma$ simultaneously on the two factors.  In the case that there are either no $1$ eigenvalues or no $-1$ eigenvalues, then there is only a single factor $\SO^{\eta(+/-)}(2m_{+/-}, \F_q)$, and no $\langle \sigma_s \rangle$ factor, and when there are no $1$ nor $-1$ eigenvalues, there are only general linear and unitary factors.  In the case $q$ is even, we only have the factor $\SO^{\eta(+)}(2m_+, \F_q)$ if there are $1$ eigenvalues.

Note that in all cases described above, if $s \in \SO^{\pm}(2n, \F_q) = H$, then we have
$$ [C_G(s) : C_H(s)] = \left\{ \begin{array}{ll} 2 & \text{if } s \text{ has any } 1 \text{ or } -1 \text{ eigenvalues,} \\ 1 & \text{otherwise.} \end{array} \right.$$
This implies that, given a semisimple class $(s)$ of $\gO^{\pm}(2n, \F_q)$ such that $s \in \SO^{\pm}(2n, \F_q)$ described by the pair of functions $(\Phi, \eta)$ as in \eqref{SemiFunctions}, this class splits into two classes of $\SO^{\pm}(2n, \F_q)$ if and only if $s$ has no $1$ nor $-1$ eigenvalues.  In the case that the class splits, then we have $s$ and $\sigma(s)$ are not conjugate in $\SO^{\pm}(2n, \F_q)$, where $\sigma(s) = hsh^{-1}$ with $h \in G \setminus H$.  

\section{Unipotent characters} \label{Unipotent}

Let $\bG$ be any reductive group over $\overline{\F}_q$ (connected or disconnected), defined over $\F_q$ with Frobenius or Steinberg morphism $F$, and let $G = \bG^F$.

First assume that $\bG$ is connected.  Let $\bT$ be a maximal $F$-stable torus of $\bG$, with $T = \bT^F$, and let $\theta$ be a linear complex character of $T$.  Let $R_T^G(\theta)$ denote the Deligne-Lusztig virtual character of $G$ corresponding to $(T, \theta)$, first defined in \cite{DeLu}.  A \emph{unipotent character} of $G$ is any irreducible character which appears as a constituent in the virtual character $R_T^G({\bf 1})$, where ${\bf 1}$ denotes the trivial character.

Consider now the case that $\bG$ is disconnected, with connected component $\bG^{\circ}$.  A \emph{unipotent character} of the group $\bG^F$ is defined to be any irreducible constituent of an induced character $\Ind_{(\bG^{\circ})^F}^{\bG^F} (\psi)$, where $\psi$ is a unipotent character of $(\bG^{\circ})^F$ as in the previous paragraph.

\subsection{Unipotent characters for special orthogonal groups} \label{UnipotentSO}

We now discuss the unipotent characters of $\SO^{\pm}(2n, \F_q)$, which are parameterized by \emph{symbols}, introduced in \cite{Lu77}.  The description of unipotent characters of $\SO^{\pm}(2n, \F_q)$ and their degrees for all $q$ is given in detail in \cite{Lu77, Lu81, As83}.  Following \cite[Section 13.8]{Ca85}, a symbol in this case is an unordered pair of finite sets of non-negative integers, $\Lambda = [\mu, \nu]$, with $\mu = (\mu_1 < \mu_2 < \cdots < \mu_r)$, $\nu = (\nu_1 < \nu_2 < \cdots < \nu_k)$, such that $\mu_1$ and $\nu_1$ are not both $0$,  $r \geq k$, and in the case $\mu = \nu$ there are two distinct symbols $\Lambda$ and $\Lambda'$ both corresponding to $[\mu, \nu]$.  The number $r-k$ is the \emph{defect} of the symbol $\Lambda$, and the \emph{rank} of the symbol $\Lambda$, which we denote by $|\Lambda|$, is defined as the non-negative integer
\begin{equation} \label{Rank}
 |\Lambda| = \sum_{i=1}^r \mu_i + \sum_{i=1}^k \nu_i - \left \lfloor \left( \frac{r+k-1}{2} \right)^2 \right \rfloor.
\end{equation}
The two symbols $\Lambda$ and $\Lambda'$ corresponding to the pair $[\mu, \nu]$ with $\mu = \nu$ (and so of defect 0) are called \emph{degenerate} symbols, while all other symbols are called \emph{non-degenerate}.  The set of all degenerate symbols (all of which have even rank and defect $0$), will be denoted by $\cG$, and $\cR$ will denote the set of all non-degenerate symbols of even defect.  We take the symbols $[\varnothing, \varnothing], [\varnothing, \varnothing]' \in \cG$ to be the only symbols of rank $0$ and even defect.

The unipotent characters of $\SO^+(2n, \F_q)$ are then parameterized by symbols of rank $n$ and with defect which is$0$ mod $4$, while unipotent characters of $\SO^-(2n, \F_q)$ are parameterized by symbols of rank $n$ and with defect which is $2$ mod $4$.  We let $\cS^+$ denote the set of all symbols of non-negative rank with defect divisible by $4$ (note $\cG \subset \cS^+$), while $\cS^-$ will denote the set of all symbols with defect congruent to $2$ mod $4$.  We define $\cS = \cS^+ \cup \cS^- = \cR \cup \cG$ to be the set of all symbols with even defect and non-negative rank.  Given $\Lambda \in \cS$ with $|\Lambda| = n > 0$, we denote by $\psi_{\Lambda}$ the unipotent character of $\SO^{\pm}(n, \F_q)$ which corresponds to the symbol $\Lambda$.

Given $\Lambda \in \cS^+$ with $|\Lambda|=n>0$, the degree of $\psi_{\Lambda}$ is given by
$$\psi_{\Lambda}(1) = (q^n - 1) \prod_{i=1}^{n-1} (q^{2i}-1) \cdot \delta(\Lambda),$$
where 
$$\delta(\Lambda) = \frac{ \prod_{i < j} (q^{\mu_i} - q^{\mu_j}) \prod_{i<j} (q^{\nu_i} - q^{\nu_j}) \prod_{i, j} (q^{\mu_i} + q^{\nu_j})}{2^{d^+(\Lambda)} q^{c(\Lambda)} \prod_{i = 1}^{r} \prod_{j = 1}^{\mu_i} (q^{2j} - 1) \prod_{i = 1}^{k} \prod_{j=1}^{\nu_i} (q^{2j} - 1)},$$
with $c(\Lambda) = \sum_{i=1}^{(r+k-2)/2} \binom{r+k-2i}{2}$, and
$$ d^+(\Lambda) = \left\{ \begin{array}{ll} (r + k - 2)/2  & \text{ if } \mu \neq \nu, \\ r = k & \text{ if } \mu = \nu. \end{array} \right. .$$
For $\Lambda \in \cS^-$ with $|\Lambda|=n >0$, the degree of $\psi_{\Lambda}$ is
$$ \psi_{\Lambda}(1) = (q^n + 1) \prod_{i=1}^{n-1} (q^{2i}-1) \cdot \delta(\Lambda),$$
where
$$\delta(\Lambda) = \frac{ \prod_{i < j} (q^{\mu_i} - q^{\mu_j}) \prod_{i<j} (q^{\nu_i} - q^{\nu_j}) \prod_{i, j} (q^{\mu_i} + q^{\nu_j})}{2^{(r+k-2)/2} q^{c(\Lambda)} \prod_{i = 1}^{r} \prod_{j = 1}^{\mu_i} (q^{2j} - 1) \prod_{i = 1}^{k} \prod_{j=1}^{\nu_i} (q^{2j} - 1)}.$$
For $\Lambda= [\varnothing, \varnothing]$ or $\Lambda=[\varnothing, \varnothing]'$, we define $\delta(\Lambda)=1$.

\subsection{Unipotent characters for orthogonal groups} \label{UnipotentO}

To describe unipotent characters of $\gO^{\pm}(2n, \F_q)$, we must describe the irreducible constituents of the induced characters $\Ind_{\SO^{\pm}(2n, \F_q)}^{\gO^{\pm}(2n, \F_q)} (\psi_{\Lambda})$ for unipotent characters $\psi_{\Lambda}$ of $\SO^{\pm}(2n, \F_q)$ just described.  Since $\gO^{\pm}(2n, \F_q) \cong \SO^{\pm}(2n, \F_q) \rtimes \langle \sigma \rangle$, this boils down to describing which unipotent characters of $\SO^{\pm}(2n, \F_q)$ are invariant under the action of $\sigma$.  The action of $\sigma$ on unipotent characters of $\SO^{\pm}(2n, \F_q)$ is given in \cite[Proposition 3.7(a)]{Ma07}.  In particular, if $\Lambda \in \cS$, then $\psi_{\Lambda}$ is invariant under the action of $\sigma$ unless $\Lambda$ is degenerate.  If $\Lambda$ is degenerate, with $\Lambda'$ the other symbol corresponding to the pair $[\mu, \nu]$ with $\mu = \nu$, then $\sigma.\psi_{\Lambda} = \psi_{\Lambda'}$.

Thus in the case $H = \SO^+(2n, \F_q)$ with $G = \gO^+(2n, \F_q)$, if $\Lambda, \Lambda' \in \cG$ are two degenerate symbols of rank $n$ corresponding to the same pair $[\mu, \nu]$ with $\mu=\nu$, then $\Ind_H^G(\psi_{\Lambda}) = \Ind_H^G(\psi_{\Lambda'})$ are the same irreducible unipotent character of $\gO^+(2n, \F_q)$, with degree twice that of $\psi_{\Lambda}$.  For all non-degenerate $\Lambda \in \cR$, with $H = \SO^{\pm}(2n, \F_q)$ and $G = \gO^{\pm}(2n, \F_q)$, we have $\Ind_H^G(\psi_{\Lambda})$ is the direct sum of two distinct irreducible extensions of $\psi_{\Lambda}$ from $H$ to $G$, each with the same degree as $\psi_{\Lambda}$.

We now define \emph{orthogonal symbols} as ordered pairs $\Xi = (\mu, \nu)$, with $\mu = (\mu_1 < \mu_2 < \cdots < \mu_r)$, $\nu = (\nu_1 < \nu_2 < \cdots < \nu_k)$, and such that $\mu_1$ and $\nu_1$ are not both $0$.  Define the defect of $\Xi = (\mu, \nu)$ as $|r-k|$, and the rank of an orthogonal symbol is still defined by \eqref{Rank}.  We let $\cO^+$ denote the orthogonal symbols with non-negative rank and defect divisible by $4$, $\cO^-$ the orthogonal symbols with defect congruent to $2$ mod $4$, and $\cO = \cO^+ \cup \cO^-$ the set of all orthogonal symbols of non-negative rank and even defect.  We have $(\varnothing, \varnothing)$ is the only orthogonal symbol of rank $0$ and even defect.  By the discussion of the previous paragraphs, the unipotent characters of $\gO^+(2n, \F_q)$ are parameterized by the orthogonal symbols in $\cO^+$ of rank $n$, and the unipotent characters of $\gO^-(2n, \F_q)$ are parameterized by those of rank $n$ in $\cO^-$.  For any $\Xi \in \cO$ of positive rank, we let $\psi_{\Xi}$ denote the corresponding unipotent character.  

When $\mu \neq \nu$, the two orthogonal symbols $(\mu, \nu)$ and $(\nu, \mu)$ correspond to the two distinct irreducible extensions of $\psi_{\Lambda}$ to the orthogonal group, where $\Lambda = [\mu, \nu] \in \cR$.  In this case, we say that $\Xi = (\mu, \nu)$ is a non-degenerate orthogonal symbol, and we write $[\Xi] = \Lambda = [\mu, \nu]$.   When $\mu = \nu \neq \varnothing$, the orthogonal symbol $(\mu, \nu)$ corresponds to the irreducible character obtained by inducing either $\psi_{\Lambda}$ or $\psi_{\Lambda'}$ where $\Lambda, \Lambda' \in \cG$ are the two degenerate symbols corresponding to $[\mu, \nu]$.  We will call $\Xi = (\mu, \nu)$ with $\mu = \nu$ a degenerate orthogonal symbol, and we write $[\Xi]$ and $[\Xi]'$ for the two degenerate symbols in $\cS^+$ corresponing to it.  We will denote by $\cR^{\sa}$ and $\cG^{\sa}$ the sets of orthogonal non-degenerate and orthogonal degenerate symbols of non-negative rank and even defect, respectively.

Given any non-degenerate orthogonal symbol $\Xi \in \cR^{\sa}$ of rank $n$, 
it follows from the above that the degree of the corresponding unipotent character of $\gO^{\pm}(2n, \F_q)$ has degree
$$ \psi_{\Xi}(1) = \psi_{[\Xi]}(1) = (q^n \mp 1) \prod_{i = 1}^{n-1} (q^{2i} - 1) \cdot \delta([\Xi]).$$
If $\Xi \in \cG^{\sa}$ is a degenerate orthogonal symbol, and $\Lambda = [\Xi]$ is one of the degenerate symbols corresponding to it, then the unipotent character $\psi_{\Xi}$ of $\gO^{+}(2n, \F_q)$ has degree
$$ \psi_{\Xi}(1) = 2 \psi_{\Lambda}(1) = (q^n - 1) \prod_{i = 1}^{n-1} (q^{2i} - 1) \cdot 2\delta(\Lambda).$$

That is, if $\Xi = (\mu, \nu) \in \cO$ is any orthogonal symbol of positive rank $n$ with $\Lambda = [\Xi] \in \cS$, then we define
$$ \delta(\Xi) = \left\{ \begin{array}{ll} \delta(\Lambda) & \text{if } \mu \neq \nu, \\
2 \delta(\Lambda) & \text{if } \mu = \nu \end{array} \right.,$$
so that we have
$$ \psi_{\Xi}(1) = (q^n \mp 1) \prod_{i = 1}^{n-1} (q^{2i} - 1) \cdot \delta(\Xi),$$
with 
$$ \delta(\Xi) = \frac{ \prod_{i < j} (q^{\mu_i} - q^{\mu_j}) \prod_{i<j} (q^{\nu_i} - q^{\nu_j}) \prod_{i, j} (q^{\mu_i} + q^{\nu_j})}{2^{(r+k-2)/2} q^{c(\Lambda)} \prod_{i = 1}^{r} \prod_{j = 1}^{\mu_i} (q^{2j} - 1) \prod_{i = 1}^{k} \prod_{j=1}^{\nu_i} (q^{2j} - 1)},$$
since in the case $\mu = \nu$ we have $d^+(\Lambda) - 1 = r - 1 = (r+k-2)/2$.  In the case $\Xi = (\varnothing, \varnothing)$, we define $\delta(\Xi) = 2$.

\subsection{Unipotent characters for centralizers of semisimple elements} \label{UnipotentCent}

We next consider the unipotent characters of the centralizers in $\gO^{\pm}(2n, \F_q)$ and $\SO^{\pm}(2n, \F_q)$ of semisimple classes $(s)$ which are contained in $\SO^{\pm}(2n, \F_q)$, as described in \eqref{CentO}.

Let $\bG = \gO_{2n}$, $G = \gO^{\pm}(2n, \F_q)$, $\bH = \SO_{2n}$, and $H= \SO^{\pm}(2n, \F_q)$.  By a slight abuse of notation, we let $F$ denote either the standard Frobenius map or the twisted Frobenius $\tF$, depending on the type $\pm$ of the orthogonal group.  If $s \in H$ is a semisimple element with conjugacy class specified by the pair $(\Phi, \eta)$ in \eqref{SemiFunctions}, then we have 
\begin{align}
(C_{\bG}(s)^{\circ})^F & = (C_{\bH}(s)^{\circ})^F  \nonumber \\ & \cong \prod_{f \in \cN(q)'} \GU(m_f, \F_{q^{\deg(f)/2}}) \times \prod_{ \{ g, g^*\}  \in \cM(q)} \GL(m_g, \F_{q^{\deg(g)}}) \label{ConnCent}\\
&  \quad \times \SO^{\eta(+)} (2m_+, \F_q) \times \SO^{\eta(-)} (2m_-, \F_q). \nonumber
\end{align}

Recall that the unipotent characters of $\GL(n, \F_q)$ and $\GU(n, \F_q)$ are each parameterized by the set of partitions of $n$, denoted by $\cP_n$.  Given a partition $\lambda \in \cP_n$ (where we also write $|\lambda|=n$), we let $\psi_{\GL, \lambda}$ and $\psi_{\GU, \lambda}$ denote the unipotent characters of $\GL(n, \F_q)$ and $\GU(n, \F_q)$, respectively, parameterized by $\lambda$.  A unipotent character $\psi^{\circ}$ of $(C_{\bG}(s)^{\circ})^F = (C_{\bH}(s)^{\circ})^F$ thus has structure given by
\begin{equation} \label{ConnUni}
\psi^{\circ} = \bigotimes_{f \in \cN'(q)}  \psi_{\GU, \lambda_{f}} \otimes \bigotimes_{\{ g, g^*\} \in \cM(q) } \psi_{\GL, \lambda_{g}} \otimes \psi_{\Lambda_+} \otimes \psi_{\Lambda_-},
\end{equation}
where $\psi_{\GU, \lambda_f}$ is a unipotent character of $\GU(m_f, \F_{q^{\deg(f)/2}})$ with $|\lambda_f| = m_f$, $\psi_{\GL, \lambda_g}$ is a unipotent character of $\GL(m_g, \F_{q^{\deg(g)}})$ with $|\lambda_g|=m_g$, and $\psi_{\Lambda_{+/-}}$ is a unipotent character of $\SO^{\eta(m_{+/-})}(2m_{+/-}, \F_q)$ with $|\Lambda_{+/-}| = m_{+/-}$ (but there is no $m_-$ or $\Lambda_-$ when $q$ is even).

To describe the unipotent characters of $C_{G}(s)$, we must describe the irreducible constituents of $\Ind_{(C_{\bG}(s)^{\circ})^F}^{C_{G}(s)} (\psi^{\circ})$.  Note that we have
\begin{align*}
\Ind_{(C_{\bG}(s)^{\circ})^F}^{C_{G}(s)} (\psi^{\circ}) & = \bigotimes_{f \in \cN'(q)}  \psi_{\GU, \lambda_{f}} \otimes \bigotimes_{\{ g, g^*\} \in \cM(q) } \psi_{\GL, \lambda_{g}} \\
& \otimes \Ind_{\SO^{\eta(+)} (2m_+, \F_q)}^{\gO^{\eta(+)} (2m_+, \F_q)}(\psi_{\Lambda_+}) \otimes \Ind_{\SO^{\eta(-)} (2m_-, \F_q)}^{\gO^{\eta(-)} (2m_-, \F_q)}(\psi_{\Lambda_-}).
\end{align*}
It follows from Section \ref{UnipotentO} that a unipotent character of $C_G(s)$ has general structure given by
\begin{equation} \label{UniCentO}
\psi = \bigotimes_{f \in \cN'(q)}  \psi_{\GU, \lambda_{f}} \otimes \bigotimes_{\{ g, g^*\} \in \cM(q) } \psi_{\GL, \lambda_{g}} \otimes \psi_{\Xi_+} \otimes \psi_{\Xi_-},
\end{equation}
where $\psi_{\Xi_{+/-}}$ is a unipotent character of $\gO^{\eta(+/-)}(2m_{+/-}, \F_q)$ with $|\Xi_{+/-}| = m_{+/-}$.

For the unipotent characters of $C_H(s)$, we must describe the irreducible consituents of $\Ind_{(C_{\bH}(s)^{\circ})^F}^{C_{H}(s)} (\psi^{\circ})$.  Given the structure of $C_H(s)$ in \eqref{CentSO}, assuming $s$ has both $1$ and $-1$ eigenvalues, we have  
\begin{align*}
\Ind_{(C_{\bH}(s)^{\circ})^F}^{C_{H}(s)} (\psi^{\circ}) & = \bigotimes_{f \in \cN'(q)}  \psi_{\GU, \lambda_{f}} \otimes \bigotimes_{\{ g, g^*\} \in \cM(q) } \psi_{\GL, \lambda_{g}} \\
& \otimes \Ind_{\SO^{\eta(+)} (2m_+, \F_q) \times \SO^{\eta(-)} (2m_-,\F_q)}^{(\SO^{\eta(+)} (2m_+, \F_q) \times \SO^{\eta(-)} (2m_-,\F_q)) \rtimes \langle \sigma_s \rangle}(\psi_{\Lambda_+} \otimes \psi_{\Lambda_-}).
\end{align*}
Because $\sigma_s$ acts as $\sigma$ simultaneously on both of the factors $\SO^{\eta(+)}(2m_+, \F_q)$ and $\SO^{\eta(-)}(2m_-, \F_q)$, we have $\psi_{\Lambda_+} \otimes \psi_{\Lambda_-}$ is invariant under $\sigma_s$ precisely when both $\Lambda_+$ and $\Lambda_-$ are non-degenerate symbols (from Section \ref{UnipotentO}).  That is, we have the following two cases:
\begin{itemize}
\item If $\Lambda_+, \Lambda_- \in \cR$, then $\psi_{\Lambda_+} \otimes \psi_{\Lambda_-}$ induces to a direct sum of two distinct characters, each of which has the same degree as $\psi_{\Lambda_+} \otimes \psi_{\Lambda_-}$.  We label these two induced characters as $\psi_{[\Lambda_+, \Lambda_-],i}$ with $i = 1, 2$.
\item If $\Lambda_+ \in \cG$ or $\Lambda_- \in \cG$, then $\psi_{\Lambda_+} \otimes \psi_{\Lambda_-}$ (and $\sigma.\psi_{\Lambda_+} \otimes \sigma.\psi_{\Lambda_-}$) induces to an irreducible character with degree twice that of $\psi_{\Lambda_+} \otimes \psi_{\Lambda_-}$.  We label this induced character as $\psi_{[\Lambda_+, \Lambda_-],0}$.
\end{itemize}
In the above, we let $[\Lambda_+, \Lambda_-]$ denote the orbit of $\sigma$ acting on both characters of the ordered pair $(\Lambda_+, \Lambda_-)$, so in the first case this orbit contains just this ordered pair, while in the second the orbit contains two.

In summary, the unipotent characters of $C_H(s)$, when $s$ has both $1$ and $-1$ eigenvalues, are all of the form
\begin{equation} \label{UniCentSO}
\omega = \bigotimes_{f \in \cN'(q)}  \psi_{\GU, \lambda_{f}} \otimes \bigotimes_{\{ g, g^*\} \in \cM(q) } \psi_{\GL, \lambda_{g}} \otimes \psi_{[\Lambda_+, \Lambda_-],i},
\end{equation}
where either $\Lambda_+ \in \cG$ or $\Lambda_- \in \cG$ with $i=0$, or $\Lambda_+, \Lambda_- \in \cR$ with $i = 1$ or $2$.  In the case $s$ does not have both $1$ and $-1$ eigenvalues (including the case that $q$ is even), then $(C_{\bH}(s)^{\circ})^F = C_H(s)$, and the unipotent characters of $C_H(s)$ are all of the form
\begin{equation} \label{UniCentSO2}
\omega = \bigotimes_{f \in \cN'(q)}  \psi_{\GU, \lambda_{f}} \otimes \bigotimes_{\{ g, g^*\} \in \cM(q) } \psi_{\GL, \lambda_{g}} \otimes \psi_{\Lambda_{+/-}}.
\end{equation}

\section{Jordan decomposition of characters for orthogonal groups} \label{JordanD}

Let $\bG$ be a connected reductive group over $\F_q$ with Frobenius morphism $F$ and $G = \bG^F$, and $\bG^*$ a dual group with dual Frobenius $F^*$ with $G^* = \bG^{*F^*}$.  Given an $F^*$-stable semisimple class $(s)$ of $\bG^*$, one may associate a conjugacy class of pairs $(\bT, \theta)$, where $\bT$ is a maximal $F$-stable torus of $\bG$ and $\theta$ is a multiplicative complex character of $\bT^F$ (see \cite[Proposition 13.12]{DiMi}).  The \emph{rational Lusztig series} of $\bG^F$ associated to a $G^*$-conjugacy class $(s)$ of a semisimple element $s \in G^*$ is the set of all irreducible characters $\chi$ of $G$ which appear as a constituent of a Deligne-Lusztig character $R_T^G(\theta)$, where $T = \bT^F$, and the $G^*$-conjugacy class $(s)$ is associated with the class of the pair $(\bT, \theta)$ as referenced above.  The rational Lusztig series corresponding to the $G^*$-semisimple class $(s)$ will be denoted by $\cE(G, (s))$.

A \emph{Jordan decomposition map} for the rational Lusztig series $\cE(G, (s))$ is a bijection
\begin{equation} \label{Jord}
J_s: \cE(G, (s)) \longrightarrow \cE(C_{\bG^*}(s)^{F^*}, 1),
\end{equation}
(after choosing some $s$ from the class $(s)$) such that, for any $\chi \in \cE(G, (s))$, if $\chi$ is a constituent of $R_T^G(\theta)$ and $\bT^*$ is the $F^*$-stable maximal torus of $\bG^*$ in duality with $\bT$, then
\begin{equation} \label{InnerProp}
\langle \chi, R_T^G(\theta) \rangle = \pm \langle J_s (\chi), R_{\bT^{*F^*}}^{C_{\bG^*}(s)^{F^*}}({\bf 1}) \rangle.
\end{equation}
Lusztig \cite{Lu84} proved the existence of a Jordan decomposition map whenever the center $Z(\bG)$ is connected, and the existence when the center is disconnected was proved by both Lusztig \cite{Lu88} and Digne and Michel \cite{DiMi90}.  We note that in the case $Z(\bG)$ is disconnected, then there will be disconnected centralizers $C_{\bG^*}(s)$, in which case the notion of unipotent character given at the beginning of Section \ref{Unipotent} is used, and we define
$$ R_{\bT^{*F^*}}^{C_{\bG^*}(s)^{F^*}}({\bf 1}) = \Ind_{(C_{\bG^*}(s)^{\circ})^{F^*}}^{C_{\bG^*}(s)^{F^*}} \left( R_{\bT^{*F^*}}^{(C_{\bG^*}(s)^{\circ})^{F^*}}({\bf 1})\right).$$

In the general case, the Jordan decomposition map also has the following property regarding character degrees (see \cite[Remark 13.24]{DiMi}).  If $J_s(\chi) = \psi$, then
\begin{equation} \label{CharDegProp}
\chi(1) = [G^*: C_{G^*}(s)]_{p'} \psi(1),
\end{equation}
where $p = \mathrm{char}(\F_q)$ and $[G^*: C_{G^*}(s)]_{p'}$ denotes the prime-to-$p$ part.

Now consider the case that $\bH = \SO_{2n}$ with Frobenius morphism given by either $F$ or $\tF$ as in Section \ref{OrthoDef}, so that $H = \SO^{\pm}(2n, \F_q)$.  In this case, we happen to have $\bH \cong \bH^*$ and $H \cong H^*$, and so we will identify $\bH$ with $\bH^*$ and $H$ with $H^*$.  We need an additional property of the Jordan decomposition in this case with respect to the action of the automorphism $\sigma$.  In particular, if $\chi \in \cE(H, (s))$, then $\sigma.\chi \in \cE(G, (\sigma(s)))$, which follows from \cite[Corollary 2.4]{NaTiTu08} or \cite[Proposition 7.2]{Tay16}.  It turns out it is very useful to choose a Jordan decomposition map which is equivariant with respect to this action of $\sigma$.  The following was first claimed in the last line of \cite[Proof of Proposition 1.7]{AuMiRo}.  In the case that $q$ is even when the center of $\SO_{2n}$ is connected, the statement follows from a result of Cabanes and Sp\"ath \cite[Theorem 3.1]{CaSp13}, and in the case $q$ is odd a proof will appear in a paper of A. Schaeffer Fry and J. Taylor \cite{SFT18}.

\begin{proposition} \label{SigmaEqui}
When $H = \SO^{\pm}(2n, \F_q) = H^*$, for any semisimple class $(s)$ of $H^*$, there exists a Jordan decomposition map (after choosing a representative $s$ of $(s)$), $J_s: \cE(H, (s)) \longrightarrow \cE(C_{H^*}(s), 1)$, such that for any $\rho \in \cE(H, (s))$ with $J_s(\rho) = \omega$, we have $J_{\sigma(s)}(\sigma.\rho) = \sigma.\omega$.
\end{proposition}

Aubert, Michel, and Rouquier \cite[Proposition 1.7]{AuMiRo} used Proposition \ref{SigmaEqui} to build a Jordan decomposition map for the groups $G = \gO^{\pm}(2n, \F_q)$, where now $\bG = \gO_{2n}$ is a disconnected group.  We give some details of this map from \cite{AuMiRo} below, and we will also prove a character degree property of this Jordan decomposition map which we need.  We take $G^* = \gO^{\pm}(2n, \F_q) = G$ and $H = \SO^{\pm}(2n, \F_q) = H^*$ as above.  Because of \eqref{OSOindex2}, any irreducible character $\chi$ of $G$ is obtained by inducing an irreducible character from $H$, where if $\rho$ is an irreducible of $H$ and $\sigma.\rho = \rho$, then $\Ind_H^G(\rho) = \chi_1 + \chi_2$ is the sum of two distinct irreducible characters of $G$, while if $\sigma.\rho \neq \rho$, then $\Ind_H^G(\rho) = \chi = \Ind_H^G(\sigma.\rho)$ is irreducible.  Given a semisimple element $s$ of $H$, we define the Lusztig series $\cE(G, (s))$ of $G$, corresponding to the $G$-conjugacy class $(s)$, as the collection of all irreducible $\chi$ of $G$ which appear as a constituent of some $\Ind_H^G(\rho)$ for some $\rho \in \cE(H, (s))$.

\begin{proposition} \label{JordanO}
For any semisimple class $(s)$ of $G = \gO^{\pm}(2n, \F_q)$ which is contained in $H = \SO^{\pm}(2n, \F_q)$, there exists a bijection
$$ J^G_s : \cE(G, (s)) \longrightarrow \cE(C_{G^*}(s), 1)$$
such that, if $\chi \in \cE(G, (s))$ with $J^G_s(\chi) = \psi$, then
\begin{align*}
\chi(1) & = \frac{2 [H^*: C_{H^*}(s)]_{p'}}{[C_{G^*}(s): C_{H^*}(s)]} \psi(1) \\
& = \left\{ \begin{array}{ll} 2[H^*: C_{H^*}(s)]_{p'} \psi(1) & \text{if } s \text{ has no } 1 \text{ nor } -1 \text{ eigenvalues,} \\
 \ [H^*: C_{H^*}(s)]_{p'} \psi(1) & \text{otherwise.} \end{array} \right.
\end{align*}
\end{proposition}
\begin{proof}  Given $s \in \SO^{\pm}(2n, \F_q)=H$, we let $J_s$ denote the Jordan decomposition map in Proposition \ref{SigmaEqui} for $H$.  We consider $\rho \in \cE(H, (s))$ and $\chi \in \cE(G, (s))$ such that $\chi$ appears in the induction of $\rho$ to $G$.  

First suppose $s$ has no $1$ nor $-1$ eigenvalues, and so as in Section \ref{CentsSemi} we have $(s)_H \neq (\sigma(s))_H$.  Since $\rho \in \cE(H, (s))$ and $\sigma.\rho \in \cE(H, (\sigma(s)))$, then we have $\rho \neq \sigma.\rho$, and so $\chi = \Ind_H^G(\rho) = \Ind_H^G(\sigma.\rho)$.  In particular, $\chi(1) = 2 \rho(1) = 2 [H^*: C_{H^*}(s)]_{p'} \psi(1)$.  We also have $C_{H^*}(s) = C_{G^*}(s)$ from Section \ref{CentsSemi}, so if $J_s(\rho) = \psi$, then $\psi \in \cE(C_{G^*}(s), 1)$.  So we may define $J_s^G(\chi) = \psi$ in this case and the statement follows. 

Now suppose $s$ has some $1$ or $-1$ eigenvalue, and we fix some $s$ in the $H$-conjugacy class $(s)_H$ such that $\sigma(s) = s$ (as in \cite[Proof of Proposition 1.7 and Sec. 1.B]{AuMiRo}).  In this case, we have $[C_{G^*}(s): C_{H^*}(s)] = 2$, and if $J_s(\rho) = \omega$, then $J_s(\sigma.\rho) = \sigma.\omega$ since $\sigma(s) = s$.  First suppose $\sigma.\rho \neq \rho$, in which case $\sigma.\omega \neq \omega$.  Then we have $\chi = \Ind_H^G(\rho) = \Ind_H^G(\sigma.\rho)$, and $\Ind_{C_{H^*}(s)}^{C_{G^*}(s)}(\omega) = \Ind_{C_{H^*}(s)}^{C_{G^*}(s)}(\sigma.\omega) =: \psi$ is irreducible, and $\psi(1) = 2 \omega(1)$.  Then we may define $J_s^G(\chi) = \psi$, and we have
$$ \chi(1) = 2 \rho(1) = 2 [H^* : C_{H^*}(s)]_{p'} \omega(1) = [H^*: C_{H^*}(s)]_{p'} \psi(1).$$
Finally, suppose $\sigma.\rho = \rho$, where we still have $\sigma(s) = s$.  Then $\Ind_H^G(\rho) = \chi + \chi'$, where $\chi$ and $\chi'$ are irrreducible, and $\chi(1) = \chi'(1) = \rho(1)$.  If $J_s(\rho) = \omega$, then $\sigma.\omega = \omega$, and $\Ind_{C_{H^*}(s)}^{C_{G^*}(s)}(\omega) = \psi + \psi'$, where $\psi$ and $\psi'$ are irreducible and $\omega(1) = \psi(1) = \psi'(1)$.  We now define $J_s^G(\chi) = \psi$ and $J_s^G(\chi') = \psi'$ (where some choice is made here), and we have
$$ \chi(1) = \rho(1) = [H^*: C_{H^*}(s)]_{p'} \omega(1) = [H^*: C_{H^*}(s)]_{p'} \psi(1),$$
which concludes the last case.
\end{proof}

\section{Character degrees and generating functions} \label{GenFuns}

The main purpose of this section is to write down expression for degrees of arbitrary irreducible characters of $\SO^{\pm}(2n, \F_q)$ and $\gO^{\pm}(2n, \F_q)$, and generating functions for the sums of these degrees.

From the previous section, the irreducible characters of $H = \SO^{\pm}(2n, \F_q) = H^*$ are in bijection with $H^*$-conjugacy classes of pairs $(s, \omega)$, where $s \in H^*$ is semisimple and $\omega$ is a unipotent character of $C_{H^*}(s)$, and if $\rho$ corresponds to the class of $(s, \omega)$, then $\rho(1) = [H^*: C_{H^*}(s)]_{p'} \omega(1)$.  By Proposition \ref{JordanO}, the irreducible characters of $G = \gO^{\pm}(2n, \F_q) = G^*$ correspond to $G^*$-classes of pairs $(s, \psi)$, where $s \in H^*$ is semisimple, $\psi$ is a unipotent character of $C_{G^*}(s)$, and if $\chi$ corresponds to the class of $(s, \psi)$, then $\chi(1) = \frac{[H^*: C_{H^*}(s)]_{p'}}{[C_{G^*}(s) : C_{H^*}(s)]} \psi(1)$.

If $\psi_{\GU, \lambda}$ is the unipotent character of $\GU(n, \F_{q^d})$ parameterized by $\lambda$, then we write
$$ \psi_{\GU, \lambda}(1) = \prod_{i=1}^n (q^{di} - (-1)^i) \cdot \delta_{\GU}(\lambda, 2d),$$
and if $\psi_{\GL, \lambda}$ is the unipotent character of $\GL(n, \F_{q^d})$ parameterized by $\lambda$, then write
$$ \psi_{\GL, \lambda}(1) = \prod_{i = 1}^n (q^{di} - 1) \cdot \delta_{\GL}(\lambda, d).$$
The expressions $\delta_{\GU}(\lambda, 2d)$ and $\delta_{\GL}(\lambda, d)$ are given explicitly in \cite[Sec. 3.3]{V17}, although we will not need them here.  When $n=0$, we take $\lambda = \varnothing$ to be the only partition of size $0$, and we define in this case $\delta_{\GU}(\lambda, 2d) = \delta_{\GL}(\lambda, d) = 1$ for any $d\geq 1$.

First consider an irreducible character $\chi$ of $G = \gO^{\pm}(2n, \F_q)$, corresponding to the class of pairs $(s, \psi)$, and the class $G^*$-class $(s)$ corresponds to the pair $(\Phi, \eta)$ satisfying \eqref{SemiFunctions}, \eqref{SemiCond1}, and \eqref{SemiCond2}.  In the general case, we have $C_{G^*}(s)$ given by \eqref{CentO}, and a unipotent character $\psi$ of that centralizer is given by \eqref{UniCentO}, where the degree of $\psi$ is given by
$$ \psi(1) = P_{\gO}(s) \prod_{f \in \cN(q)'} \delta_{\GU}(\lambda_f, \deg(f)) \prod_{\{ g, g^*\} \in \cM(q)} \delta_{\GL}(\lambda_g, \deg(g)) \cdot \delta(\Xi_+) \delta(\Xi_-),$$
where $P_{\gO}(s)$ is the expression
\begin{align*}
P_{\gO}(s) = &\prod_{f \in \cN(q)'}  \prod_{i = 1}^{m_f} (q^{i \deg(f)/2} - (-1)^i) \prod_{ \{ g, g^*\} \in \cM(q)} \prod_{i = 1}^{m_g} (q^{i \deg(g)} - 1) \\
& \cdot (q^{m_+} - \eta(+)1) \prod_{i = 1}^{m_+ - 1} (q^{2i} - 1) \cdot (q^{m_-} - \eta(-)1) \prod_{i=1}^{m_- - 1} (q^{2i} - 1),
\end{align*}
where either of the last two terms involving $m_+$ and $m_-$ are left off when $m_+ = 0$ or $m_-=0$, respectively.

From \eqref{CentSO} and the sentences following it, we have
$$ [H^*: C_{H^*}(s)]_{p'} = \left\{ \begin{array}{ll} \frac{(q^n \mp 1) \prod_{i = 1}^{n-1} (q^{2i}-1)}{2P_{\gO}(s)} & \text{if } m_+ \neq 0 \text{ and } m_- \neq 0, \\ \frac{(q^n \mp 1) \prod_{i = 1}^{n-1} (q^{2i}-1)}{P_{\gO}(s)} & \text{otherwise,} \end{array} \right.$$
where we note in the first case we always have $p \neq 2$.  From this and Proposition \ref{JordanO}, if the character $\chi$ of $\gO^{\pm}(2n, \F_q)$ corresponds to $(s, \psi)$, then the \emph{modified} character degree of $\chi$ is given by
$$\frac{\chi(1)}{2 (q^n \mp 1) \prod_{i=1}^{n-1} (q^{2i} - 1)} = \prod_{f \in \cN(q)'} \delta_{\GU}(\lambda_f, \deg(f)) \prod_{\{ g, g^*\} \in \cM(q)} \delta_{\GL}(\lambda_g, \deg(g)) \cdot$$
\begin{equation} \label{ModCharDegO}
\cdot \left\{ \begin{array}{ll} 1 & \text{if } s \text{ has no } 1 \text{ nor } -1 \text{ eigenvalues,} \\
\frac{1}{2}  \delta(\Xi_{+/-}) & \text{if } s \text{ has } 1 \text{ or } -1 \text{ eigenvalues, but not both,} \\
\frac{1}{4} \delta(\Xi_{+}) \delta(\Xi_{-}) & \text{otherwise, } \end{array} \right.
\end{equation}
where the last case only occurs for $p \neq 2$.

Based on the above, we define
\begin{equation} \label{Vdefn}
\sV(z) = \sum_{\Xi \in \cO} \frac{1}{2} \delta(\Xi) z^{|\Xi|}.
\end{equation}
Note that the constant term of $\sV(z)$ is $1$, since for the only $\Xi= (\varnothing, \varnothing) \in \cO$ of rank $0$, we have defined $\delta(\Xi) = 2$.

We also define $N^*(q; d)$ to be the number of polynomials in $\cN(q)$ of degree $d$, and $M^*(q; d)$ to be the number of unordered pairs $\{ g, g^* \}$ in $\cM(q)$ such that $\deg(g) = d$.  Recall $\Sigma(G)$ denotes the sum of the irreducible character degrees of a finite group $G$, and that $e=e(q)$ is defined to be $2$ if $q$ is odd and $1$ if $q$ is even.

\begin{proposition} \label{GenFunO}  For any $n >0$, the expression
$$\frac{\Sigma(\gO^{+}(2n, \F_q))}{2(q^n -1) \prod_{i=1}^{n-1} (q^{2i} - 1) }  + \frac{\Sigma(\gO^{-}(2n, \F_q))} {2(q^n + 1) \prod_{i=1}^{n-1} (q^{2i} - 1)}$$
is the coefficient of $z^n$, in the generating function
$$\prod_{d \geq 1} \left( \sum_{\lambda \in \cP} \delta_{\GU}(\lambda, 2d) z^{|\lambda| d} \right)^{N^*(q; 2d)} \left (\sum_{\lambda \in \cP} \delta_{\GL}(\lambda, d) z^{|\lambda| d} \right)^{M^*(q; d)} \sV(z)^e,$$
which has constant term $1$, and where $e=e(q)$ is as in \eqref{edefn}.
\end{proposition}
\begin{proof}  The generating function we seek has, as the coefficient of $z^n$ for $n>0$, the sum of the modified character degrees given in \eqref{ModCharDegO} over all possible classes of pairs $(s, \psi)$ which parameterize irreducible characters of $\gO^{+}(2n, \F_q)$ or $\gO^{-}(2n, \F_q)$.  When specifying the semisimple class $(s)$ by choosing the functions $\Phi$ and $\eta$ in \eqref{SemiFunctions} which satisfy \eqref{SemiCond1}, \eqref{SemiCond2}, and \eqref{SemiCond3}, first note that if $m_+ = m_- = 0$, then the character corresponding to $(s, \psi)$ depends only on \eqref{SemiCond1} and the choice of partitions $\lambda_f$ of $m_f$ for $f \in \cN(q)'$ and $\lambda_g$ of $m_g$ for $\{g, g^*\} \in \cM(q)$.  If $m_+$ or $m_-$ is nonzero, then as in \eqref{UniCentO} we must choose unipotent characters $\psi_{\Xi_{+/-}}$, where $\Xi_{+/-}$ are orthogonal symbols of rank $m_{+/-}$, and with defect modulo $4$ depending on the signs from the function $\eta$ satisfying \eqref{SemiCond3}.  However, if we only specify $\Phi$ satisfying \eqref{SemiCond1}, and choose partitions $\lambda_f$ of $m_f$, $\lambda_g$ of $m_g$, and orthogonal symbols $\Xi_{+/-} \in \cO$ of rank $m_{\pm}$, then the function $\eta$ will then be determined by the defects of the chosen orthogonal symbols $\Xi_{+/-}$.  This determines a pair $(s, \psi)$ corresponding to a unique character of $\gO^+(2n, \F_q)$ or $\gO^-(2n, \F_q)$.  That is, an irreducible character of $\gO^+(2n, \F_q)$ or $\gO^-(2n, \F_q)$ and its degree is completely determined by a choice of the function $\Phi$ satisfying \eqref{SemiCond1}, partitions $\lambda_f$ of $m_f$ and $\lambda_g$ of $m_g$, and orthogonal symbols $\Xi_{+/-}$ of rank $m_{+/-}$.

Given the above argument together with the expression for the modified character degree in \eqref{ModCharDegO}, the generating function we want is given by
$$ \prod_{f \in \cN(q)'} \left( \sum_{\lambda_f \in \cP} \delta_{\GU} (\lambda_f, \deg(f)) z^{|\lambda_f| \deg(f)/2} \right)$$
$$ \cdot \prod_{ \{ g, g^*\} \in \cM(q)} \left( \sum_{\lambda_g \in \cP} \delta_{\GL} (\lambda_g, \deg(g)) z^{|\lambda_g| \deg(g)} \right) $$
$$\cdot \left( \sum_{\Xi_+ \in \cO} \frac{1}{2}\delta(\Xi_+) z^{|\Xi_+|} \right)  \left( \sum_{\Xi_- \in \cO} \frac{1}{2}\delta(\Xi_-) z^{|\Xi_-|} \right),$$
where the last factor is not included in the case $p=2$.  The last two factors can thus be replaced by $\sV(z)^e$, while in the first two factors the expressions only depend on the degrees of the polynomials $f(t)$ and $g(t)$, and not the polynomials themselves.  The generating function can thus be written as
$$\prod_{d \geq 2} \left( \sum_{\lambda \in \cP} \delta_{\GU}(\lambda, d) z^{|\lambda| d/2} \right)^{N^*(q; d)} \prod_{d \geq 1} \left (\sum_{\lambda \in \cP} \delta_{\GL}(\lambda, d) z^{|\lambda| d} \right)^{M^*(q; d)} \sV(z)^e,$$
where we note $\cN(q)'$ contains no polynomials of degree $1$.  As stated in Section \ref{CentsSemi}, all polynomials in $\cN(q)'$ have even degree, and so we may replace $d$ with $2d$ in the first product, and index it over all $d \geq 1$.  Noting that the constant term of this generating function is $1$ since the same is true of each factor, the result follows.
\end{proof}

We next consider the degrees of characters of $\SO^{\pm}(2n, \F_q)$.  Let $\rho$ be such a character, and suppose it corresponds to the pair $(s, \omega)$ via the Jordan decomposition.  First suppose that the semisimple class $(s)$ of $H^* = \SO^{\pm}(2n, \F_q)$ does not contain both $1$ and $-1$ eigenvalues.  
In this case, as in Section \ref{CentsSemi}, the centralizer $C_{H^*}(s)$ has structure
\begin{align*}
C_{H^*}(s) \cong & \prod_{f \in \cN(q)'} \GU(m_f, \F_{q^{\deg(f)/2}}) \times \prod_{ \{ g, g^*\}  \in \cM(q)} \GL(m_g, \F_{q^{\deg(g)}}) \\
& \times \SO^{\eta(+/-)} (2m_{+/-}, \F_q), 
\end{align*}
where the last factor does not appear if $m_+ = m_- = 0$.   Then we have $[H^*: C_{H^*}(s)]_{p'} = \frac{(q^n \mp 1) \prod_{i = 1}^{n-1} (q^{2i} - 1)}{P_{\gO}(s)}$, as above.  A unipotent character $\omega$ of $C_{H^*}(s)$ here is given by \eqref{UniCentSO2} with degree
$$ \omega(1) = P_{\gO}(s) \prod_{f \in \cN(q)'} \delta_{\GU}(\lambda_f, \deg(f)) \prod_{\{ g, g^*\} \in \cM(q)} \delta_{\GL}(\lambda_g, \deg(g)) \cdot \delta(\Lambda_{+/-}).$$
By \eqref{CharDegProp}, the modified character degree of $\rho$ in this case is given by
$$\frac{\rho(1)}{(q^n \mp 1) \prod_{i=1}^{n-1} (q^{2i} - 1)} = \prod_{f \in \cN(q)'} \delta_{\GU}(\lambda_f, \deg(f)) \prod_{\{ g, g^*\} \in \cM(q)} \delta_{\GL}(\lambda_g, \deg(g)) \cdot$$
\begin{equation} \label{ModCharDegSO1}
\cdot \left\{ \begin{array}{ll} 1 & \text{if } s \text{ has no } 1 \text{ nor } -1 \text{ eigenvalues,} \\
\delta(\Lambda_{+/-}) & \text{if } s \text{ has } 1 \text{ or } -1 \text{ eigenvalues, but not both.}  \end{array} \right.
\end{equation}

Now consider the case that the semisimple class $(s)$ has both $1$ and $-1$ eigenvalues (so $q$ is necessarily odd).  The centralizer $C_{H^*}(s)$ is then given by \eqref{CentSO}, and $[H^*: C_{H^*}(s)]_{p'} = \frac{(q^n \mp 1) \prod_{i = 1}^{n-1} (q^{2i} - 1)}{2 P_{\gO}(s)}$ as above.  An arbitrary unipotent character $\omega$ of $C_{H^*}(s)$ is given by \eqref{UniCentSO}.  By the description in Section \ref{UnipotentCent}, the degree of the factor $\psi_{[\Lambda_+, \Lambda_-], i}$ is given by
$$ \psi_{[\Lambda_+, \Lambda_-], i}(1) = \left\{ \begin{array}{ll} 2\psi_{\Lambda_+}(1) \psi_{\Lambda_-}(1) & \text{ if } i = 0, \text{ that is, if } \Lambda_+ \in \cG \text{ or } \Lambda_- \in \cG, \\
\psi_{\Lambda_+}(1) \psi_{\Lambda_-}(1) & \text{ if } i= 1, 2, \text{ that is, if } \Lambda_+, \Lambda_- \in \cR. \end{array} \right.$$
That is, the degree of $\omega$ is given by
\begin{align*}
\omega(1) = P_{\gO}(s)  & \prod_{f \in \cN(q)'} \delta_{\GU}(\lambda_f, \deg(f)) \prod_{\{ g, g^*\} \in \cM(q)} \delta_{\GL}(\lambda_g, \deg(g)) \cdot \\
& \cdot \left\{ \begin{array}{ll} 2\delta(\Lambda_+) \delta(\Lambda_-) & \text{if } \Lambda_+ \in \cG \text{ or } \Lambda_- \in \cG \\
\delta(\Lambda_+) \delta(\Lambda_-) & \text{ if } \Lambda_+, \Lambda_- \in \cR. \end{array} \right.
\end{align*}

It follows from these facts and \eqref{CharDegProp} that the modified character degree of $\rho$ in this case is given by
$$\frac{\rho(1)}{(q^n \mp 1) \prod_{i=1}^{n-1} (q^{2i} - 1)} = \prod_{f \in \cN(q)'} \delta_{\GU}(\lambda_f, \deg(f)) \prod_{\{ g, g^*\} \in \cM(q)} \delta_{\GL}(\lambda_g, \deg(g)) \cdot$$
\begin{equation} \label{ModCharDegSO2}
\cdot \left\{ \begin{array}{ll} \delta(\Lambda_+) \delta(\Lambda_-) & \text{if } \Lambda_+ \in \cG \text{ or } \Lambda_- \in \cG\\
\frac{1}{2}\delta(\Lambda_+) \delta(\Lambda_-) & \text{if } \Lambda_+, \Lambda_- \in \cR. \end{array} \right.
\end{equation}
We recall that in the first case above, when $\Lambda_+ \in \cG$ or $\Lambda_- \in \cG$, the same character is obtained if we choose $\Lambda_+'$ and $\Lambda_-'$, where $\Lambda' = \Lambda$ when $\Lambda \in \cR$.  In the second case, there are two distinct characters for every choice of $\Lambda_+, \Lambda_- \in \cR$.

Since the set of orthogonal symbols is the disjoint union $\cO = \cR^{\sa} \cup \cG^{\sa}$ of orthogonal non-degenerate and degenerate symbols, then if we define
$$\sR(z)_{\sa} = \sum_{\Xi \in \cR_{\sa}} \frac{1}{2} \delta(\Xi) z^{|\Xi|} \quad \text{and} \quad \sG(z)_{\sa} = \sum_{\Xi \in \cG_{\sa}} \frac{1}{2} \delta(\Xi) z^{|\Xi|},$$
then we have $\sV(z) = \sR(z)_{\sa} + \sG(z)_{\sa}$.
Now define
$$ \sR(z) = \sum_{\Lambda \in \cR} \frac{1}{2}\delta(\Lambda) z^{|\Lambda|} \quad \text{and} \quad \sG(z) = \sum_{\Lambda \in \cG} \frac{1}{2} \delta(\Lambda) z^{|\Lambda|}.$$
As in Sections \ref{UnipotentSO} and \ref{UnipotentO}, for each non-degenerate symbol $\Lambda = [\mu, \lambda] \in \cR$, there are two non-degenerate orthogonal symbols  $\Xi = (\mu, \lambda) \in \cR_{\sa}$ and $\Xi' = (\lambda, \mu) \in \cR_{\sa}$, where $\delta(\Xi) = \delta(\Xi') = \delta(\Lambda)$.  It follows that $2\sR(z) = \sR(z)_{\sa}$.  For each pair of degenerate symbols $\Lambda, \Lambda' \in \cG$ corresponding to $[\mu, \lambda]$ with $\mu = \lambda$, there is a single degenerate orthogonal symbol $\Xi = (\mu, \lambda) \in \cG_{\sa}$, and we have $\delta(\Xi) = 2\delta(\Lambda)$ in this case.  It follows that $\sG(z)_{\sa} = \sG(z)$.  Now we have
\begin{equation} \label{SOunipsum}
\sum_{\Lambda \in \cS} \delta(\Lambda) z^{|\Lambda|} = 2\sR(z) + 2\sG(z) = \sR(z)_{\sa} + \sG(z)_{\sa} + \sG(z) = \sV(z) + \sG(z).
\end{equation}
We note that $\sV(z) + \sG(z)$ has constant term $2$, since $\sV(z)$ and $\sG(z)$ each have constant term $1$.

\begin{proposition} \label{GenFunSO} The expression
$$\frac{\Sigma(\SO^{+}(2n, \F_q))}{(q^n -1) \prod_{i=1}^{n-1} (q^{2i} - 1) }  + \frac{\Sigma(\SO^{-}(2n, \F_q))} {(q^n + 1) \prod_{i=1}^{n-1} (q^{2i} - 1)}$$
is the coefficient of $z^n$, $n>0$, in the generating function
$$\prod_{d \geq 1}  \left( \sum_{\lambda \in \cP} \delta_{\GU}(\lambda, 2d) z^{|\lambda| d} \right)^{N^*(q; 2d)} \left (\sum_{\lambda \in \cP} \delta_{\GL}(\lambda, d) z^{|\lambda| d} \right)^{M^*(q; d)}  (\sV(z)^e + \sG(z)^e),$$
which has constant term 2, and where $e=e(q)$ is as in \eqref{edefn}.
\end{proposition}
\begin{proof}  We prove the statement for the cases that $q$ is even or odd separately.  Given all possible pairs $(s, \omega)$ which correspond to irreducible characters of $\SO^+(2n, \F_q)$ or $\SO^-(2n, \F_q)$, we need the sum of the modified character degrees to be the coefficient of $z^n$ in our generating function.

First suppose that $q$ is even.  Consider a character corresponding to the pair $(s, \omega)$.  Assume first that $(s)$ is a semisimple class which has $1$ as an eigenvalue.  Similar to the proof of Proposition \ref{GenFunO}, once we specify the function $\Phi$ in \eqref{SemiFunctions}, and the unipotent factor $\psi_{\Lambda}$ for some $\Lambda \in \cS$ with $|\Lambda| = m_+$, then the sign $\eta$ in \eqref{SemiFunctions} is determined, as is the sign corresponding to the type of special orthogonal group with a character to which $(s, \omega)$ corresponds.  In the case $(s)$ has no $1$ eigenvalues, then we recall that there are two semisimple classes corresponding to the choice of $\Phi$ as in the last paragraph of Section \ref{CentsSemi}, and the sign of the corresponding special orthogonal group is determined by \eqref{SemiCond3}.  Thus in this case, the choice of the function $\Phi$ and the unipotent character $\omega$ determines two distinct characters of $\SO^{\pm}(2n, \F_q)$ of the same degree.  It follows from this and \eqref{ModCharDegSO1} that in the case that $q$ is even, the generating function we seek is given by
$$ \prod_{f \in \cN(q)'} \left( \sum_{\lambda_f \in \cP} \delta_{\GU} (\lambda_f, \deg(f)) z^{|\lambda_f| \deg(f)/2} \right)$$
$$ \cdot \prod_{ \{ g, g^*\} \in \cM(q)} \left( \sum_{\lambda_g \in \cP} \delta_{\GL} (\lambda_g, \deg(g)) z^{|\lambda_g| \deg(g)} \right) \cdot \left( \sum_{\Lambda \in \cS} \delta(\Lambda) z^{|\Lambda|} \right),$$
noting that the constant term in the last factor is $2$, corresponding to the two irreducible characters corresponding to the choice of $\Phi$ and $\omega$ when $(s)$ has no $1$ eigenvalues.  From the argument in Proposition \ref{GenFunO}, together with \eqref{SOunipsum}, this simplifies to the claimed generating function,
$$\prod_{d \geq 1}  \left( \sum_{\lambda \in \cP} \delta_{\GU}(\lambda, 2d) z^{|\lambda| d} \right)^{N^*(q; 2d)} \left (\sum_{\lambda \in \cP} \delta_{\GL}(\lambda, d) z^{|\lambda| d} \right)^{M^*(q; d)}  (\sV(z) + \sG(z)).$$

We now assume $q$ is odd.  As in the previous arguments of this section, the modified degree of an irreducible character of $\SO^{\pm}(2n, \F_q)$ boils down to the choice of $\Phi$ and a unipotent character $\omega$.  It follows as before that the generating function we need is of the form
$$\prod_{d \geq 1}  \left( \sum_{\lambda \in \cP} \delta_{\GU}(\lambda, 2d) z^{|\lambda| d} \right)^{N^*(q; 2d)} \left (\sum_{\lambda \in \cP} \delta_{\GL}(\lambda, d) z^{|\lambda| d} \right)^{M^*(q; d)} \cdot \sF(z),$$
where $\sF(z)$ corresponds to the choice of the unipotent factor of $\omega$ coming from the $1$ and $-1$ eigenvalues of $(s)$.  Note that $\sF(z)$ should have constant term $2$, since there are two distinct semisimple classes $(s)$ corresponding to $\Phi$ when $s$ has no $1$ or $-1$ eigenvalues.  We now describe the factor $\sF(z)$.

If we are choosing the unipotent factor corresponding to $(s)$ having a nonzero number  of both $1$ and $-1$ eigenvalues, then we consider the terms in \eqref{ModCharDegSO2}.  For every choice with $\Lambda_+ \in \cG$ or $\Lambda_- \in \cG$, we get the same character for the choices $\Lambda_+'$ and $\Lambda_-'$, and so a factor of $1/2$ should be included in the generating function for the corresponding sums accounting for these terms.  On the other hand, when $\Lambda_+, \Lambda_- \in \cR$, we obtain two characters of the same degree corresponding to this choice, but this factor of $2$ when considering the resulting modified character degree is cancelled by the factor of $1/2$ in \eqref{ModCharDegSO2}.  This suggests that $\sF(z)$ should have the following form:
\begin{align}
\sF(z) & = \frac{1}{2} \left[ \left(\sum_{\Lambda_+ \in \cG} \delta(\Lambda_+) z^{|\Lambda_+|} \right) \left( \sum_{\Lambda_- \in \cR} \delta(\Lambda_-) z^{|\Lambda_-|} \right)  \right. \label{FFactor}\\ 
& \quad \quad +  \left(\sum_{\Lambda_+ \in \cR} \delta(\Lambda_+) z^{|\Lambda_+|} \right) \left( \sum_{\Lambda_- \in \cG} \delta(\Lambda_-) z^{|\Lambda_-|} \right) \nonumber \\
& \quad \quad \left. + \left( \sum_{\Lambda_+ \in \cG} \delta(\Lambda_+) z^{|\Lambda_+|} \right) \left( \sum_{\Lambda_- \in \cG} \delta(\Lambda_-) z^{|\Lambda_-|}\right) \right] \nonumber \\
& \quad \quad + \left( \sum_{\Lambda_+ \in \cR} \delta(\Lambda_+) z^{|\Lambda_+|} \right)\left( \sum_{\Lambda_- \in \cR} \delta(\Lambda_-) z^{|\Lambda_-|} \right). \nonumber
\end{align}
That is, all terms obtained in the above by taking non-constant terms in each sum account for the unipotent factor when there are both $1$ and $-1$ eigenvalues.  We now explain why \eqref{FFactor} also accounts for the cases that $(s)$ does not have both $1$ and $-1$ eigenvalues.  As already mentioned, for the case that $s$ has neither $1$ nor $-1$ eigenvalues we only need that $\sF(z)$ has constant term $2$.  This follows from the fact that each sum in \eqref{FFactor} indexed by $\cG$ has constant term $2$, and each indexed by $\cR$ has constant term $0$.  In the case that $s$ has either $1$ or $-1$ eigenvalues, but not both, we must consider the factor \eqref{ModCharDegSO1}.  We can have $\Lambda_+ \in \cR$ or $\Lambda_- \in \cR$, and these terms are obtained by considering the first two products of sums in \eqref{FFactor}, by taking the product of the terms in the sums of $\cR$ with the constant term for the sums over $\cG$.  Since the constant term for sums indexed by $\cG$ is $2$, the factor of $1/2$ in $\sF(z)$ yields coefficients given by the corresponding factor in \eqref{ModCharDegSO1}.  If $\Lambda_+ \in \cG$ or $\Lambda_- \in \cG$, these terms are obtained from the third product of sums in \eqref{FFactor} by taking the constant term in one factor, and again the $1/2$ in front gives terms in \eqref{FFactor}.  Note that since both sums in the fourth product of sums in \eqref{FFactor} has 0 as the constant term, all resulting term coefficients have already been accounted for in the case $s$ has both $1$ and $-1$ eigenvalues.  Thus $\sF(z)$ given by \eqref{FFactor} is the needed factor in the generating function.

From \eqref{FFactor} and the definitions of $\sG(z)$ and $\sR(z)$, we now have
$$\sF(z) = \frac{1}{2} \left[ 2 \sG(z) \cdot 2 \sR(z) + 2 \sR(z) \cdot 2\sG(z) + 2 \sG(z) \cdot 2 \sG(z))\right] + 2 \sR(z) \cdot 2 \sR(z)$$
$$ = 4 \sR(z)^2 + 4 \sG(z) \sR(z) +  \sG(z)^2 +  \sG(z)^2 = (2 \sR(z) + \sG(z))^2 + \sG(z)^2$$
$$ = \sV(z)^2 + \sG(z)^2.$$
The result in the case that $q$ is odd follows.
\end{proof}

The following result is obtained in \cite[Proposition 4.1]{V17}, which we state here since we will apply it in several calculations in the next section.

\begin{proposition} \label{OldResult}
We have the following identity of power series, where $e=e(q)$ is as in \eqref{edefn}:
\begin{align*}
\prod_{d \geq 1} & \left( \sum_{\lambda \in \cP} \delta_{\GU}(\lambda, 2d) z^{|\lambda| d} \right)^{N^*(q; 2d)}  \left (\sum_{\lambda \in \cP} \delta_{\GL}(\lambda, d) z^{|\lambda| d} \right)^{M^*(q; d)} \\
& = \frac{1}{1-z} \prod_{i \geq 1} \frac{(1 - z/q^{2i-1})^e}{1 -z^2/q^{2i}} \prod_{1 \leq i < j \atop{ i + j \text{ odd}}} (1 - z^2/q^{i+j})^e.
\end{align*}
\end{proposition}

\section{Main results} \label{MainRs}

We may now prove the main results of the paper.  In the following, we apply the generating functions in the previous section to give an infinite product expression for $\sV(z)$, which was defined in \eqref{Vdefn} to have $1/2$ of the sum of all modified character degrees of unipotent characters of $\gO^{+}(2n, \F_q)$ and $\gO^{-}(2n, \F_q)$ as the coefficient of $z^n$ for $n>0$, and constant term $1$.

\begin{theorem} \label{GenFunUnipO}
For any prime power $q$, we have
$$\sV(z) = \prod_{i \geq 1} \frac{1 + z/q^{2i-1}}{1 - z/q^{2i-1}} \prod_{1 \leq i < j \atop{ i + j \text{ odd}}} \frac{1}{1 - z^2/q^{i+j}}.$$
\end{theorem}
\begin{proof}  When $q$ is odd, we have by Propositions \ref{IndicatorsO}(i) and \ref{GenFunO} with $e=2$ that
\begin{align} 
\prod_{d \geq 1} & \left( \sum_{\lambda \in \cP} \delta_{\GU}(\lambda, 2d) z^{|\lambda| d} \right)^{N^*(q; 2d)} \left (\sum_{\lambda \in \cP} \delta_{\GL}(\lambda, d) z^{|\lambda| d} \right)^{M^*(q; d)} \sV(z)^2 \nonumber\\
& = \frac{1}{1-z} \prod_{i \geq 1} \frac{(1 + z/q^{2i-1})^2}{1 - z^2/q^{2i}}. \label{StartO}
\end{align}
Substituting the identity in Proposition \ref{OldResult} with $e=2$ into \eqref{StartO} and simplifying gives
$$\prod_{i \geq 1} (1 - z/q^{2i-1})^2 \prod_{1 \leq i < j \atop{ i + j \text{ odd}}} (1 - z^2/q^{i+j})^2 \cdot \sV(z)^2 = \prod_{i \geq 1} (1 + z/q^{2i-1})^2.$$
Solving for $\sV(z)$ and noting that the coefficients must be positive yields the claim when $q$ is odd.  Since the unipotent characters have degrees which are the same expressions in $q$ whether $q$ is even or odd, then the result also holds for $q$ even.
\end{proof}

We may immediately apply Theorem \ref{GenFunUnipO} to show that every irreducible character of  $\gO^{\pm}(2n, \F_q)$ has Frobenius-Schur indicator $1$ when $q$ is even.

\begin{theorem} \label{OisTO}
When $q$ is a power of $2$, the groups $\gO^{\pm}(2n, \F_q)$ are totally orthogonal for all $n$.
\end{theorem}
\begin{proof}  By Propositions \ref{GenFunO} and \ref{OldResult} with $e=1$, together with Theorem \ref{GenFunUnipO}, we have that for $n>0$,
$$\frac{\Sigma(\gO^{+}(2n, \F_q))}{2(q^n -1) \prod_{i=1}^{n-1} (q^{2i} - 1) }  + \frac{\Sigma(\gO^{-}(2n, \F_q))} {2(q^n + 1) \prod_{i=1}^{n-1} (q^{2i} - 1)}$$
is the coefficient of $z^n$ in the generating function
$$ \frac{1}{1-z} \prod_{i \geq 1} \frac{1 - z/q^{2i-1}}{1 -z^2/q^{2i}} \prod_{1 \leq i < j \atop{ i + j \text{ odd}}} (1 - z^2/q^{i+j})\prod_{i \geq 1} \frac{1 + z/q^{2i-1}}{1 - z/q^{2i-1}} \prod_{1 \leq i < j \atop{ i + j \text{ odd}}} \frac{1}{1 - z^2/q^{i+j}}$$
$$ = \frac{1}{1-z} \prod_{i \geq 1} \frac{1 + z/q^{2i-1}}{1 - z^2/q^{2i}}.$$
The result now follows from Proposition \ref{IndicatorsO}(ii).
\end{proof}

R\"{a}m\"{o} proved \cite[Theorem 1.1]{Ra11} that when $q$ is even, all elements of $\SO^{\pm}(4m, \F_q)$ are strongly real.  In particular, since $\gO^{\pm}(4m, \F_q)$ is totally orthogonal and $\SO^{\pm}(4m, \F_q)$ is a real group, then the next result follows immediately from Lemma \ref{index2}(i).

\begin{theorem} \label{SO4misTO}
When $q$ is a power of $2$, the groups $\SO^{\pm}(4m, \F_q)$ are totally orthogonal for all $m$.
\end{theorem}

We finally have the following result, which follows from Theorem \ref{SO4misTO} and the other cases of strongly real finite simple groups, summarized in \cite[Theorem 3.2]{TV17}.

\begin{theorem} \label{SimpleTO}
Let $G$ be a finite simple group.  Then $G$ is totally orthogonal if and only if $G$ is strongly real.
\end{theorem}

The classification of all finite simple groups which are real is given by Tiep and Zalesski \cite{TZ05}, and the classification of strongly real finite simple groups completed in \cite{VG10, Ra11} yields that every real finite simple group is also strongly real.  We obtain the following interesting behavior of the Frobenius-Schur indicators of characters of finite simple groups.

\begin{cor} \label{SimpleFSCor}
Let $G$ be a finite simple group.  If $G$ has an irreducible character $\chi$ such that $\vep(\chi) = -1$, then $G$ also has an irreducible character which is not real-valued.
\end{cor}
\begin{proof}  Since every finite simple group which is real is also strongly real, then from Theorem \ref{SimpleTO} every real finite simple group is totally orthogonal.  In other words, there are no finite simple groups which are real and which have irreducible characters with Frobenius-Schur indicator $-1$.  The result follows.
\end{proof}

We may also find infinite product expressions for the generating functions $\sG(z)$ and $\sR(z)$ by applying Proposition \ref{GenFunSO} and Theorem \ref{GenFunUnipO}, as follows.

\begin{theorem} \label{GenFunSO}
For any prime power $q$, we have
$$ \sG(z) = \prod_{1 \leq i < j \atop{ i + j \text{ odd}}} \frac{1}{1 - z^2/q^{i+j}}, \text{ and  }$$
$$ \sR(z) = \left[ \left(\prod_{i \geq 1} \frac{1 + z/q^{2i-1}}{1 - z/q^{2i-1}} \right) - 1 \right] \prod_{1 \leq i < j \atop{ i + j \text{ odd}}} \frac{1}{1 - z^2/q^{i+j}}.$$
\end{theorem}
\begin{proof} For $q$ odd, by Propositions \ref{GenFunSO} and \ref{OldResult} with $e=2$, together with Proposition \ref{IndicatorsSO}(i), we have
\begin{align*}
\frac{1}{1-z} \prod_{i \geq 1} & \frac{(1 - z/q^{2i-1})^2}{1 -z^2/q^{2i}} \prod_{1 \leq i < j \atop{ i + j \text{ odd}}} (1 - z^2/q^{i+j})^2 \cdot (\sV(z)^2 + \sG(z)^2) = \\
 & = \frac{1}{1-z} \prod_{i \geq 1} \frac{(1 + z/q^{2i-1})^2}{1 - z^2/q^{2i}} + \frac{1}{1-z} \prod_{i \geq 1} \frac{(1 - z/q^{2i-1})^2}{1 - z^2/q^{2i}}.
\end{align*}
After substituting in the expression in Theorem \ref{GenFunUnipO} for $\sV(z)$ and some simplification, we obtain
$$\prod_{i \geq 1} (1 - z/q^{2i-1})^2 \prod_{1 \leq i < j \atop{ i + j \text{ odd}}} (1 - z^2/q^{i+j})^2 \cdot \sG(z)^2 =  \prod_{i \geq 1} (1 - z/q^{2i-1})^2.$$
Solving for $\sG(z)$ gives the result for $q$ odd.  Since $\sG(z)$ is defined in terms of unipotent character degrees which are expressions in $q$ independent of the parity of $q$, the result holds for all $q$.  The expression for $\sR(z)$ is obtained from the fact that $\sV(z) = \sR(z) + \sG(z)$.
\end{proof}

We note that the expressions for $\sV(z)$ and $\sG(z)$ given in Theorems \ref{GenFunUnipO} and \ref{GenFunSO} strongly resemble the identity in \cite[Theorem 4.1]{V17} for the generating function for the sum of the modified character degrees of the unipotent characters of the symplectic and odd-dimensional special orthogonal groups (which we will apply in Section \ref{symplectic}).  There should exist some universal symmetric function proof of these identities, by utilizing the formulas for these modified character degrees in terms of hook-lengths of symbols given in \cite{Ol86, Ma95}.

Next, we describe the twisted Frobenius-Schur indicators for the groups $\SO^{\pm}(4m+2, \F_q)$.  The next result is proved in the case that $q$ is odd in \cite[Theorem 5.1(ii)]{TV17}, and so the statement holds for all $q$.  

\begin{theorem} \label{SOIndicators1}
Let $q$ be a power of $2$.  For any $m \geq 0$ and any irreducible character $\chi$ of $\SO^{\pm}(4m+2, \F_q)$, we have $\eps(\chi) = 1$ and $\vep(\chi) \geq 0$.
\end{theorem}
\begin{proof} By Propositions \ref{GenFunSO} and \ref{OldResult} with $e = 1$, we have that when $q$ is even and $n >0$, the expression 
$$\frac{\Sigma(\SO^{+}(2n, \F_q))}{(q^n -1) \prod_{i=1}^{n-1} (q^{2i} - 1) }  + \frac{\Sigma(\SO^{-}(2n, \F_q))} {(q^n + 1) \prod_{i=1}^{n-1} (q^{2i} - 1)}$$
is the coefficient of $z^n$ in the generating function
\begin{equation} \label{SOsimplify}
\frac{1}{1-z} \prod_{i \geq 1} \frac{1 - z/q^{2i-1}}{1 -z^2/q^{2i}} \prod_{1 \leq i < j \atop{ i + j \text{ odd}}} (1 - z^2/q^{i+j}) \cdot (\sV(z) + \sG(z)).
\end{equation}
Substituting the expressions for $\sV(z)$ in Theorem \ref{GenFunUnipO} and for $\sG(z)$ in Theorem \ref{GenFunSO}, \eqref{SOsimplify} becomes 
$$\frac{1}{1-z} \prod_{i \geq 1} \frac{1 + z/q^{2i-1}}{1 - z^2/q^{2i}} + \frac{1}{1-z} \prod_{i \geq 1} \frac{1 - z/q^{2i-1}}{1 - z^2/q^{2i}}.$$
The fact that $\eps(\chi) = 1$ for any irreducible character of $\SO^{\pm}(4m+2, \F_q)$, $m \geq 0$, now follows from Proposition \ref{IndicatorsSO}(ii).  In particular, we have $\sigma(h)$ is conjugate to $h^{-1}$ in $\SO^{\pm}(4m+2, \F_q)$ for every $h \in \SO^{\pm}(4m+2, \F_q)$.  Since $\gO^{\pm}(4m+2, \F_q)$ is totally orthogonal by Theorem \ref{OisTO}, it follows from Lemma \ref{index2}(ii) that $\vep(\chi) \geq 0$ for all irreducible characters $\chi$ of $\SO^{\pm}(4m+2, \F_q)$.
\end{proof}

\section{Symplectic groups} \label{symplectic}

As an application of the results in the previous section, in this section we give a shorter and more direct proof of the main result of \cite{V05} regarding the finite symplectic groups $\Sp(2n, \F_q)$ with $q$ odd.  For any $q$, define $\Sp_{2n}$ to be the collection of invertible transformations on $V = \overline{\F}_q^{2n}$ which stabilize the symplectic form $\langle \cdot, \cdot \rangle$ defined by 
$$ \langle v, w \rangle = {^\top v} \left( \begin{array}{rr}    & -I_n \\ I_n &   \end{array} \right) w.$$
Taking $F$ to be the standard Frobenius, define $\Sp(2n, \F_q) = \Sp_{2n}^F$ (and $\Sp(0, \F_q)$ is taken to be the group with $1$ element).  Now let 
$$ y = \left( \begin{array}{rr} I_n &    \\    & -I_n \end{array} \right),  $$
which is \emph{skew-symplectic}, meaning we have $\langle yv, yw \rangle = - \langle v, w \rangle$ for all $v, w \in V$.  Let $\iota$ be the order $2$ automorphism on $\Sp(2n, \F_q)$ defined by ${^\iota g} = y g y^{-1}$.  The result \cite[Theorem 1.3]{V05} states that when $q$ is odd, every irreducible complex character $\chi$ of $\Sp(2n, \F_q)$ satisfies $\epi(\chi) = 1$.  We give a condensed proof of this result by applying the methods and results obtained in the previous sections of this paper.  

It follows \cite[Proposition 3.2 and proof of Corollary 6.1]{V05} that the number of elements in $\Sp(2n, \F_q)$ which satisfy ${^\iota g} = g^{-1}$ when $q$ is odd is given by $\frac{|\Sp(2n, \F_q)|}{|\GL(n, \F_q)|}$, where $|\Sp(2n, \F_q)| = q^{n^2} \prod_{i = 1}^n (q^{2i}-1)$.  Now note that
\begin{equation} \label{SpGoal}
\sum_{n \geq 0} \frac{q^{n^2}}{|\GL(n, \F_q)|} z^n = \sum_{n \geq 0} \frac{z^n}{ \prod_{i = 1}^n (1 - 1/q^i)} = \prod_{i \geq 1} (1 - z/q^{i-1})^{-1},
\end{equation}
where the last equality follows from an identity of Euler (see \cite[(2.2.5)]{An76}).  Now, to show that $\epi(\chi) = 1$ for all irreducible $\chi$ of $\Sp(2n, \F_q)$ for all $n$, we must show that $\frac{\Sigma(\Sp(2n, \F_q))}{\prod_{i = 1}^n (q^{2i} - 1)}$ is the coefficient of $z^n$ in $\prod_{i \geq 1} (1 - z/q^{i-1})^{-1}$, where $\Sigma(\Sp(2n, \F_q))$ is the sum of the character degrees of $\Sp(2n, \F_q)$.  We now compute this character degree sum directly.

If $\bG = \Sp_{2n}$, then a dual group is given by $\bG^* = \SO_{2n+1}$, which may be defined as the group of invertible transformations on $\overline{\F}_q^{2n+1}$ which stabilize the quadratic form 
$$ Q_1(v) = x_1 x_{2n} + \cdots + x_n x_{n+1} + x_{2n+1}^2,$$
where $v = (x_1, \ldots, x_{2n+1})$.  Then $F^*$ can be taken to be the standard Frobenius, and $\bG^{*F^*} = \SO_{2n+1}^{F^*} = \SO(2n+1, \F_q)$.  We take $G = \Sp(2n, \F_q)$ and $G^* = \SO(2n+1, \F_q)$ for the rest of this section.

Let $(s)$ be any semisimple class of $G^*$.  In this case, the class is determined by a pair of functions 
\begin{equation} \label{SpSemiFunctions}
\Phi: \cN(q) \cup \cM(q) \rightarrow \Z_{\geq 0},  \quad \eta: \{ t+1 \} \rightarrow \{ +, - \}, 
\end{equation}
such that, writing $\Phi(f) = m_f$ for $f \in \cN(q)'$, $\Phi(\{ g, g^* \}) = m_g$ for $\{ g, g^* \} \in \cM(q)$, $\Phi(t + 1) = m_{+}$, $\Phi(t-1)=m_-$, and $\eta(t+1) = \eta$ , we have
\begin{equation} \label{SpSemiCond1}
|\Phi| := \sum_{f \in \cN'(q)} m_f \deg(f)/2 + \sum_{ \{g, g^* \} \in \cM(q) } m_g \deg(g) + m_+ + m_- = n,
\end{equation}
where
\begin{equation} \label{SpSemiCond2}
\eta= + \, \text{ if } \, m_{+} = 0, 
\end{equation}
which follows from \cite[Sec. 2.6 Case (C)]{Wall} and \cite[Proposition 16.8]{AuMiRo}.  The centralizer in $G^*$ of an element from this class has structure
\begin{align}
C_{G^*}(s) \cong \prod_{f \in \cN(q)'} \GU(m_f, \F&_{q^{\deg(f)/2}}) \times \prod_{ \{ g, g^*\}  \in \cM(q)} \GL(m_g, \F_{q^{\deg(g)}}) \label{CentSp}\\
& \times \gO^{\eta} (2m_+, \F_q) \times \SO (2m_- +1, \F_q), \nonumber
\end{align}
see \cite[Sec. 1B]{AuMiRo} and \cite[Lemma 2.2]{Ng10}.  Then we have
\begin{equation} \label{index} 
[G^*: C_{G^*}(s)]_{p'} = \left\{ \begin{array}{ll} \frac{\prod_{i=1}^{n}(q^{2i}-1)}{P_{\Sp}(s)} & \text{if } m_+ =0, \\ \frac{\prod_{i=1}^{n}(q^{2i}-1)}{2P_{\Sp}(s)} & \text{if } m_+ \neq 0, \end{array} \right.
\end{equation}
where 
\begin{align}
P_{\Sp}(s) = &\prod_{f \in \cN(q)'}  \prod_{i = 1}^{m_f} (q^{i \deg(f)/2} - (-1)^i) \prod_{ \{ g, g^*\} \in \cM(q)} \prod_{i = 1}^{m_g} (q^{i \deg(g)} - 1) \label{Psp} \\
& \cdot (q^{m_+} - \eta 1) \prod_{i = 1}^{m_+ - 1} (q^{2i} - 1) \cdot \prod_{i=1}^{m_-} (q^{2i} - 1), \nonumber
\end{align}
and the last two products involving $m_+$ or $m_-$ are taken to be $1$ if $m_+=0$ or $m_-=0$, respectively

The unipotent characters of $\SO(2n+1, \F_q)$ are parameterized by the symbols of Section \ref{UnipotentSO}, but with odd defect.  That is, we consider symbols which are a pair $\Upsilon = [\mu, \nu]$ with $\mu = (\mu_1 < \mu_2 < \cdots < \mu_r)$, $\nu = (\nu_1 < \nu_2 < \cdots <\nu_k)$, such that $\mu_1$ and $\nu_1$ are not both $0$, and such that the defect $r-k > 0$ is odd.  The rank $|\Upsilon|$ of such a symbol is still defined as in \eqref{Rank}, and the symbols of rank $n$ with odd defect parameterize the unipotent characters of $\SO(2n+1, \F_q)$.  If $\psi_{\Upsilon}$ is the unipotent character of $\SO(2n+1, \F_q)$ parameterized by $\Upsilon$, then the degree of $\psi_{\Upsilon}$ may be written as
$$ \psi_{\Upsilon}(1) = \prod_{i = 1}^{n} (q^{2i} - 1) \cdot \delta(\Upsilon),$$
where $\delta(\Upsilon)$ is an expression which we will not specify here (see \cite[Sec. 13.8]{Ca85}).  We will let $\cS_{\text{odd}}$ denote the set of all symbols of non-negative rank and odd defect.

Similar to Section \ref{UnipotentCent}, a unipotent character of \eqref{CentSp} is of the form
$$\psi = \bigotimes_{f \in \cN'(q)}  \psi_{\GU, \lambda_{f}} \otimes \bigotimes_{\{ g, g^*\} \in \cM(q) } \psi_{\GL, \lambda_{g}} \otimes \psi_{\Xi} \otimes \psi_{\Upsilon},$$
where $\psi_{\Xi}$ is a unipotent character of $\gO^{\eta}(2m_+, \F_q)$, so $\Xi$ is an orthogonal symbol with $|\Xi|=m_+$, and $\Upsilon$ is a symbol of odd defect with $|\Upsilon| = m_-$.  The degree of this character is given by
\begin{equation} \label{UniSpDeg}
\psi(1) = P_{\Sp}(s) \prod_{f \in \cN(q)'} \delta_{\GU}(\lambda_f, \deg(f)) \prod_{\{ g, g^*\} \in \cM(q)} \delta_{\GL}(\lambda_g, \deg(g)) \cdot \delta(\Xi) \delta(\Upsilon).
\end{equation}

Now let $\chi$ be an irreducible character of $\Sp(2n, \F_q)$ corresponding to the class of pairs $(s, \psi)$ from the Jordan decomposition map in Section \ref{JordanD}.  From \eqref{CharDegProp}, \eqref{index}, and \eqref{UniSpDeg}, we have
\begin{align*}
\frac{\chi(1)}{\prod_{i=1}^{2n} (q^{2i} - 1)} = \prod_{f \in \cN(q)'}  & \delta_{\GU}(\lambda_f, \deg(f)) \prod_{\{ g, g^*\} \in \cM(q)} \delta_{\GL}(\lambda_g, \deg(g)) \\ & \cdot \left\{ \begin{array}{ll} \delta(\Upsilon) & \text{if } m_+ = 0 ,\\ \frac{1}{2} \delta(\Xi) \delta(\Upsilon) & \text{if } m_+ \neq 0 \end{array} \right.
\end{align*}

Now define
$$ \sW(z) = \sum_{\Upsilon \in \cS_{\text{odd}}} \delta(\Upsilon) z^{|\Upsilon|},$$
which is taken to have constant term $1$, corresponding to the the only symbol $[(0), \varnothing]$ of rank $0$ and odd defect.  By essentially the same argument as in the proof of Proposition \ref{GenFunO}, we have that when $q$ is odd, $\frac{\Sigma(\Sp(2n, \F_q))}{\prod_{i=1}^n (q^{2i} - 1)}$ is the coefficient of $z^n$ in the generating function 
\begin{equation} \label{GenFunSp}
\prod_{d \geq 1} \left( \sum_{\lambda \in \cP} \delta_{\GU}(\lambda, 2d) z^{|\lambda| d} \right)^{N^*(q; 2d)} \left (\sum_{\lambda \in \cP} \delta_{\GL}(\lambda, d) z^{|\lambda| d} \right)^{M^*(q; d)} \sV(z) \sW(z).
\end{equation}
By \cite[Theorem 4.1]{V17}, we have
$$ \sW(z) = \prod_{i \geq 1} \frac{1 + z/q^{2i}}{1 - z/q^{2i-1}} \prod_{1 \leq i < j \atop{ i + j \text{ odd}}} \frac{1}{1 - z^2/q^{i+j}}.$$
Now substitute this, along with the identity from Proposition \ref{OldResult} with $e=2$, and the expression for $\sV(z)$ from Theorem \ref{GenFunUnipO}, into \eqref{GenFunSp} and simplify.  This gives us $\frac{\Sigma(\Sp(2n, \F_q))}{\prod_{i=1}^n (q^{2i}-1)}$ is the coefficient of $z^n$ in the generating function
\begin{align*}
\frac{1}{1-z} & \prod_{i \geq 1} \frac{1}{1 - z^2/q^{2i}} \prod_{i \geq 1} (1 + z/q^{2i}) \prod_{i \geq 1} (1 + z/q^{2i-1}) \\&  = \frac{1}{1-z} \prod_{i \geq 1} \frac{1}{(1 - z/q^i)(1 + z/q^i)} \prod_{i \geq 1}(1 + z/q^i) = \prod_{i \geq 1} (1 - z/q^{i-1})^{-1}.
\end{align*}
Since this matches \eqref{SpGoal}, we have now recovered the result that $\epi(\chi)=1$ for all irreducible characters $\chi$ of $\Sp(2n, \F_q)$ whenever $q$ is odd.

 \section{A twisted version of strong reality} \label{Twizted}

In this section we investigate a twisted version of strong reality in certain groups of Lie type in the following sense.  If $G$ is a finite group and $\iota$ is an automorphism of $G$ such that $\iota^2 = 1$, then we say $g \in G$ is \emph{strongly $\iota$-real} if there is an element $y \in G$ with the property $\iota(y) = y^{-1}$, such that $y$ conjugates $g$ to $\iota(g)^{-1}$.  This is equivalent to $g$ being the product of two elements, $g = y_1 y_2$, such that $\iota(y_i) = y_i^{-1}$ for $i = 1,2$.  We find that for a large family of classical and conformal groups $G$ over a finite field, there exists such an automorphism $\iota$ such that every element of $G$ is strongly $\iota$-real, and every $\chi \in \Irr(G)$ satisfies $\epi(\chi) = 1$.  

\subsection{Special orthogonal groups}
We consider the groups $G =\SO^{\pm}(4m+2, \F_q)$ with the order $2$ automorphism $\sigma$, where $\sigma(g) = hgh^{-1}$ and $h \in \gO^{\pm}(4m+2, \F_q) \setminus \SO^{\pm}(4m+2, \F_q)$.  As mentioned in Section \ref{InvolsSec}, it was proved in \cite[Theorem 5.1(ii)]{TV17} that when $q$ is odd we have $\eps(\chi) = 1$ for all $\chi \in \Irr(G)$, and in the case that $q$ is even this is obtained in Theorem \ref{SOIndicators1} above.  

When $q$ is odd, it follows from \cite[Lemma 4.7]{SFV16} that for any $g \in G = \SO^{\pm}(4m+2, \F_q)$, there exists some $k \in \gO^{\pm}(4m+2, \F_q) \setminus \SO^{\pm}(4m+2, \F_q)$ such that $k^2 = 1$ and $k g k^{-1} = g^{-1}$.  It follows that $y = hk \in G$, and so $\sigma(y) = kh = y^{-1}$, and $yg y^{-1} = \sigma(g)^{-1}$.  That is, when $q$ is odd we have every element of $G$ is strongly $\sigma$-real.  We now prove this statement for the case that $q$ is even, which we can obtain by slightly extending the methods of R\"{a}m\"{o} from \cite{Ra11}.

Let $q$ be a power of $2$, and take $\tilde{V} = \F_q^{2n}$.  We may take the symplectic group $\Sp(2n, \F_q)$ to be the group of transformations on $\tilde{V}$ which stabilize the non-degenerate symplectic form $\langle \cdot, \cdot \rangle$ defined by the form at the beginning of Section \ref{symplectic} but restricted to $\F_q$-scalars.  Then either of the orthogonal groups $\gO^{\pm}(2n, \F_q)$ may be embedded in $\Sp(2n, \F_q)$ as follows.  There exists a quadratic form $\tilde{Q}$ on $V$ for which $\gO^{\pm}(2n, \F_q)$ is the stabilizer, and such that $\tilde{Q}(v + w) = \tilde{Q}(v) + \tilde{Q}(w) + \langle v, w \rangle$ for all $v, w \in V$.  Given any $g \in \Sp(2n, \F_q)$, we say that a subspace $W$ of $\tilde{V}$ is \emph{symplectically indecomposable} with respect to $g$ if $W$ has no direct sum decomposition which is orthogonal with respect to $\langle \cdot, \cdot \rangle$, such that the summands are non-trivial and $g$-invariant.  Then $\tilde{V}$ can be orthogonally decomposed as $\tilde{V} = \oplus V_i$, into symplectically indecomposable $g$-invariant subspaces $V_i$.  As outlined in \cite[Section 1]{Gow81}, it follows from results of Huppert \cite{Hu80} that each $V_i$ must be of one of the following forms:
\begin{enumerate}
\item[(1)] $V_i$ is a (non-orthogonal) direct sum, $V_i = U_i \oplus W_i$, of degenerate totally isotropic $g$-invariant cyclic subspaces (with $\dim_{\F_q}(U_i) = \dim_{\F_q}(W_i)$), and $g$ has a single elementary divisor $(t-1)^j$ on both $U_i$ and $W_i$;
\item[(2)] $V_i$ is $g$-cyclic, such that $g$ has a single elementary divisor $f(t)^j$, where $f(t)$ is irreducible and self-reciprocal, and if $f(t) = t-1$ then $j$ is even;
\item[(3)] $V_i$ is $g$-cyclic, and $V_i$ is a (non-orthogonal) direct sum $V_i = U_i \oplus W_i$ of totally isotropic $g$-invariant subspaces (with $\dim_{\F_q}(U_i) = \dim_{\F_q}(W_i)$), such that $g$ has a single elementary divisor $f(t)^j$ on $U_i$, and a single elementary divisor $f^*(t)^{j}$ on $W_i$, such that $f(t)$ is irreducible and $f(t) \neq f^*(t)$.
\end{enumerate}
In case (1), we say that $V_i$ is \emph{bicyclic} with respect to $g$.  Given this decomposition $\tilde{V} = \oplus V_i$, we write $g_i = g|_{V_i}$ and $g = \oplus g_i$.  We note that when $g \in \gO^{\pm}(2n, \F_q)$, we have $g|_{V_i}$ stabilizes the quadratic form $\tilde{Q}$ restricted to $V_i$.  We write $\gO_{V_i}$ and $\SO_{V_i}$ for the orthogonal and special orthogonal groups on $V_i$ defined by the quadratic form $\tilde{Q}|_{V_i}$.  In particular, if $g \in \gO^{\pm}(2n, \F_q)$, then $g_i \in \gO_{V_i}$ for any $V_i$ in the decomposition above.

A key result we need is the following.  For any $g \in \gO^{\pm}(2n, \F_q)$ with $q$ even, we have $g \in \SO^{\pm}(2n, \F_q)$ if and only if $\mathrm{rank}(g+I)$ is even, by \cite[Proposition 3.2]{Ra11}.  We use this result along with other results of R\"{a}m\"{o} from \cite{Ra11} in the next result.

\begin{proposition} \label{sigmareal} If $q$ is a power of $2$, then every element of $\SO^{\pm}(4m+2, \F_q)$ is strongly $\sigma$-real.
\end{proposition}
\begin{proof}  It is enough to show that for any $g \in \SO^{\pm}(4m+2, \F_q)$, there is some $k \in \gO^{\pm}(4m+2, \F_q) \setminus \SO^{\pm}(4m+2, \F_q)$ such that $k^2 = 1$ and $k$ conjugates $g$ to $g^{-1}$.

We take $\tilde{V} = \oplus V_i$ as in the decomposition with respect to $g$ given above, and $g = \oplus g_i$.  We have $\dim(\tilde{V}) = 4m+2$, and each $V_i$ is either cyclic or bicyclic with respect to $g$.  Note that we must have each $\dim(V_i)$ even, since otherwise $V_i$ is cyclic of odd dimension which is not possible.  We consider each possibility for $V_i$.

First suppose $4|\dim(V_i)$.  If $V_i$ is cyclic, then by \cite[Proposition 3.3]{Ra11} there is an involution $k_i \in \SO_{V_i}$ which inverts $g_i$ (that is, which conjugates $g_i$ to $g_i^{-1}$).  If $V_i$ is bicyclic, then by \cite[Proposition 3.5]{Ra11} there is an involution $k_i \in \SO_{V_i}$ which inverts $g_i$.

Now suppose $\dim(V_i) \equiv 2($mod $4)$.  If $V_i$ is cyclic, then by \cite[Proposition 3.3]{Ra11}, there is an involution $k_i \in \gO_{V_i} \setminus \SO_{V_i}$ which inverts $g_i$.  If $V_i$ is bicyclic, then by \cite[Propositions 3.16 and 3.17]{Ra11}, there exists an involution $k_i \in \SO_{V_i}$ and an involution $k_i' \in \gO_{V_i} \setminus \SO_{V_i}$, each of which inverts $g_i$.

Since $\dim(\tilde{V})=4m+2$, then there are an odd number of $V_i$'s such that $\dim(V_i) \equiv 2($mod $4)$.  So, when taking $k = \oplus k_i$, we may choose an odd number of the $k_i$'s such that $\mathrm{rank}(k_i + I_{V_i})$ is odd.  Thus $\mathrm{rank}(k + I)$ is odd, and so $k$ inverts $g$,  $k^2 = 1$, and $k \in \gO^{\pm}(4m+2, \F_q) \setminus \SO^{\pm}(4m+2, \F_q)$.
\end{proof}

\subsection{Conformal groups} 
Now take $V=\overline{\F}_q^{2n}$, and let $Q$ be the quadratic form on $V$ as defined in Section \ref{OrthoDef}.  The group of transformations in $\GL_{2n}$ which stabilize $Q$ up to a scalar factor is denoted $\CO_{2n} = \CO_{2n}(Q) = \CO(2n, \overline{\F}_q)$, and is called the \emph{conformal orthogonal group} with respect to $Q$.  So for each $g \in \CO_{2n}$, there is a scalar $\beta(g) \in \overline{\F}_q^{\times}$ such that, for all $v \in V$, we have $Q(gv) = \beta(g) Q(v)$.  The function 
$$ \beta: \CO_{2n} \rightarrow \overline{\F}_q^{\times} $$
is a multiplicative character, called the \emph{similitude} character.  Thus $\gO_{2n} = \mathrm{ker}(\beta)$.

Like the orthogonal group $\gO_{2n}$, the group $\CO_{2n}$ is also disconnected, and the connected component $\CO^{\circ}_{2n}$ is the \emph{special conformal orthogonal group}, also denoted $\CSO_{2n} = \CSO(2n, \overline{\F}_q)$.  We may again define an automorphism $\sigma$ on $\CO_{2n}$ and $\CSO_{2n}$, which is defined by conjugation by the element $h \in \gO_{2n} \setminus \SO_{2n}$ specified in Section \ref{OrthoDef}.  If $F$ is the standard Frobenius endomorphism, define $\tilde{F} = \sigma F$ to be a twisted Frobenius map on $\CO_{2n}$ and $\CSO_{2n}$.  Then we define the finite groups
$$ \CO^+(2n, \F_q) = \CO_{2n}^F, \quad \CSO^+(2n, \F_q) = \CSO_{2n}^F, $$
$$\CO^-(2n, \F_q) = \CO_{2n}^{\tilde{F}}, \quad \text{and} \quad \CSO^-(2n, \F_q) = \CSO_{2n}^{\tilde{F}}.$$
Then $\CSO^{\pm}(2n, \F_q)$ is an index 2 subgroup of $\CO^{\pm}(2n, \F_q)$.  We note that the groups $\CSO^{\pm}(2n, \F_q)$ modulo their center, $\PCSO^{\pm}(2n, \F_q)$, are the finite groups of Lie type corresponding to the simple algebraic groups of adjoint type of type $D_n$ (for $+$-type) and ${^2 D_n}$ (for $-$-type).  

The similitude character $\beta$ restricts to a character on the groups $\CO^{\pm}(2n, \F_q)$ and $\CSO^{\pm}(2n, \F_q)$, with kernel $\gO^{\pm}(2n, \F_q)$ and $\SO^{\pm}(2n, \F_q)$, respectively, and image $\F_q^{\times}$.  In particular, $\gO^{\pm}(2n, \F_q)$  and  $\SO^{\pm}(2n, \F_q)$ are normal subgroups of index $q-1$ in $\CO^{\pm}(2n, \F_q)$ and $\CSO^{\pm}(2n, \F_q)$, respectively.  Note that for any scalar matrix $\alpha I$ with $\alpha \in \F_q^{\times}$, we have $\alpha I \in \CO^{\pm}(2n, \F_q)$ with $\beta(\alpha I) = \alpha^2$.  In particular, if we denote by $Z$ the collection of such scalar matrices, we have $\beta(Z)$ is the set of squares in $\F_q^{\times}$.  If $q$ is even, then every $\alpha \in \F_q^{\times}$ is a square, and so we have $\beta(Z) = \F_q^{\times}$.  It follows that when $q$ is even, we have the directs products
$$ \CO^{\pm}(2n, \F_q) = Z \times \gO^{\pm}(2n, \F_q), \quad \CSO^{\pm}(2n, \F_q) = Z \times \SO^{\pm}(2n, \F_q).$$ 
In the case that $q$ is odd, we have
$$ Z \cdot \gO^{\pm}(2n, \F_q) \quad \text{and} \quad Z \cdot \SO^{\pm}(2n, \F_q)$$
are the index 2 subgroups of $\CO^{\pm}(2n, \F_q)$ and $\CSO^{\pm}(2n, \F_q)$, respectively, consisting of elements with square similitude.  We mention that in the case $q$ is odd, we may also describe $\CSO^{\pm}(2n, \F_q)$ as
$$\CSO^{\pm}(2n, \F_q) = \{ g \in \CO^{\pm}(2n, \F_q) \, \mid \, \det(g) = \beta(g)^n \},$$
and for any $q$ we have $Z \leq \CSO^{\pm}(2n, \F_q)$.

Define an automorphism $\iota$ on $\CSO^{\pm}(2n, \F_q)$ as follows.  If $n=2m$ is even, then define 
$$ \iota(g) = \beta(g)^{-1} g,$$
while if $n = 2m+1$ is odd, define
$$ \iota(g) = \beta(g)^{-1} \sigma(g) = \beta(g)^{-1} hgh^{-1}.$$
On the group $\CO^{\pm}(2n, \F_q)$, we let $\iota$ denote the automorphism $\iota(g) = \beta(g)^{-1} g$ for any $n$.

\begin{proposition}  \label{CSOtwist} Let $q$ be any prime power, and let $\iota$ be the automorphism of $G=\CSO^{\pm}(2n, \F_q)$ or $G=\CO^{\pm}(2n, \F_q)$ defined above.  Then every element of $G$ is strongly $\iota$-real, and any $\chi \in \Irr(G)$ satisfies $\epi(\chi) = 1$.
\end{proposition}
\begin{proof} These statements in most cases follow from known results or from results in this paper.   Write $K = \SO^{\pm}(2n, \F_q)$ if $G = \CSO^{\pm}(2n, \F_q)$, and $K = \gO^{\pm}(2n, \F_q)$ if $G = \CO^{\pm}(2n, \F_q)$.  We first consider the statement on elements in $G$ being $\iota$-real.  

If $q$ is even, let $g \in G$, and write $g = \alpha k$ with $\alpha \in \F_q^{\times}$ and $k \in K$, so that $\mu(g) = \alpha^2$.  Then $\iota(g)^{-1} = \alpha \sigma(k)^{-1}$ if $K = \SO^{\pm}(2n, \F_q)$ and $n$ is odd, and $\iota(g)^{-1} = \alpha k^{-1}$ otherwise.  The statement now follows from Proposition \ref{sigmareal} in the first case, and the fact that $\gO^{\pm}(2n, \F_q)$ and $\SO^{\pm}(4m, \F_q)$ are strongly real \cite{ElNo82, Gow81, Ra11} in the other cases.  In the case that $q$ is odd, the statement is exactly \cite[Theorem 1]{V06} if $G = \CO^{\pm}(2n, \F_q)$ and is \cite[Proposition 7.2]{RV18} if $G = \CSO^{\pm}(2n, \F_q)$.

For the statement on characters, first consider the case $q$ is even.  Since $G = Z \times K$, then any $\chi \in \Irr(G)$ is of the form $\chi = \theta \otimes \psi$ with $\theta \in \Irr(Z)$ and $\rho \in \Irr(K)$.  Writing $g = \alpha k$ with $\alpha \in Z$ and $k \in K$, then $\chi(g) = \theta(\alpha) \rho(k)$, and $\iota(g) = \iota(\alpha) \iota(k) = \alpha^{-1} \iota(k)$, where $\iota(k) = \sigma(k)$ if $n$ is odd and $K = \SO^{\pm}(2n, \F_q)$, and $\iota(k) = k$ otherwise.  Now we have
\begin{align*}
\epi(\chi) & = \frac{1}{|G|} \sum_{g \in G} \chi(g \iota(g)) = \frac{1}{|Z|} \sum_{\alpha \in Z} \theta(\alpha \iota(\alpha)) \frac{1}{|K|} \sum_{k \in K} \rho(k \iota(k)) \\
& = \frac{1}{|Z|} \sum_{\alpha \in Z} \theta(\alpha \alpha^{-1}) \frac{1}{|K|} \sum_{k \in K} \rho(k \iota(k)) = \epi(\rho).
\end{align*}
Now $\epi(\chi) = \epi(\rho) = 1$ by Theorems \ref{OisTO}, \ref{SO4misTO}, and \ref{SOIndicators1}.  

If $q$ is odd, in the case $G = \CO^{\pm}(2n, \F_q)$ the fact that $\epi(\chi) = 1$ for all $\chi \in \Irr(G)$ is exactly \cite[Theorem 2]{V06}.  The only statement left to prove is that $\epi(\chi) = 1$ for all $\chi \in \Irr(G)$ with $G = \CSO^{\pm}(2n, \F_q)$, and we follow the same argument given in the proof of \cite[Theorem 2]{V06}.  Let $\chi \in \Irr(G)$ with $(\pi, W)$ an irreducible representation of $G$ with character $\chi$.  Since we know $\iota(g)$ is conjugate to $g^{-1}$ in $G$ for all $g \in G$ from above, then $\iota.\pi \cong \hat{\pi}$.  Then there is a non-degenerate bilinear form $B$ on $W$ with the property \eqref{BiFormInd}, and by \eqref{BiFormInd2} we must show that this $B$ is symmetric.  Let $H = Z \cdot K$, which is an index $2$ subgroup of $G$, and so $\pi$ restricted to $H$ is either an irreducible representation $(\pi', W)$, or a direct sum of two non-isomorphic irreducible representations $(\pi'_1, W_1)$ and $(\pi'_2, W_2)$ of $H$.  In the first case, $(\pi', W)$ restricted to $K = \SO^{\pm}(2n, \F_q)$ is an irreducible representation $\phi$.  Note that for any $k \in K$, $\iota(k) = k$ if $n$ is even and $\iota(k) = \sigma(k)$ if $n$ is odd.  So, for any $k \in K$ and $v, w \in W$ we have
$$B(\iota.\pi(k)v, \pi(k)w) = B(v, w) = B(\iota.\phi(k)v, \phi(k)w).$$
Since $\epi(\phi) = 1$ by \cite[Theorem 2]{Gow85} and \cite[Theorem 5.1(ii)]{TV17}, then there is a non-degenerate form $B'$ on $W$, unique up to scalar, such that
$$B'(\iota.\phi(k)v, \phi(k)w) = B'(v, w),$$
for all $k \in K$ and $v, w \in W$, and which is symmetric.  Thus $B$ must be a scalar multiple of $B'$, and since $B'$ is symmetric, then so is $B$.  Thus $\epi(\chi) = 1$ in this case.

In the case that $\pi$ restricted to $H$ is a direct sum of the non-isomorphic irreducible representations $(\pi'_1, W_1)$ and $(\pi'_2, W_2)$, these representations restrict to irreducible representations $(\phi_1, W_1)$ and $(\phi_2, W_2)$ of $K$.  As above, we have $\epi(\phi_1)=1$, so there is a non-degenerate form $B_1$ on $V_1$, unique up to scalar multiple, such that 
$$B_1(\iota.\phi_1(k)v, \phi_1(k)w) = B_1(v,w), \; \text{ for all } \; v, w \in W_1, k \in K,$$
and $B_1$ is symmetric.  Next we claim that the form $B$ on $W$ is non-degenerate on $W_1$.  Suppose that $B$ is non-degenerate on $W_1 \times W_2$.  Then we have, for any $k \in K$, $v_1 \in W_1$, $v_2 \in W_2$,
$$ B(\iota.\pi(k) v_1, \pi(k)v_2) =  B(\iota.\phi_1(k) v_1, \phi_2(k) v_2) =B(v_1, v_2).$$
With the assumption that $B$ is non-degenerate on $W_1 \times W_2$, this implies that $\iota.\phi_1 \cong \hat{\phi_2}$.  Then we have $\iota.\phi_1 \cong \hat{\phi_1}$ already, and so $\phi_1 \cong \phi_2$.  The central characters of $\pi_1$ and $\pi_2$ both agree with that of $\pi$, and so we have $\pi_1 \cong \pi_2$.  However we already know that $\pi_1$ and $\pi_2$ are non-isomorphic, and so it follows from Schur's Lemma that we must have $B$ is $0$ on $W_1 \times W_2$.  Since $B$ is non-degenerate on $W$ and $W = W_1 \oplus W_2$, it follows that $B$ must be non-degenerate on $W_1$.  But now for any $k \in K$ and any $v, w \in W_1$ we have 
$$ B(\iota.\phi_1(k)v, \phi_1(k)w) = B(v, w),$$
which means $B$ must be equal to $B_1$ on $W_1$ up to a scalar multiple.  Since $B_1$ is symmetric on $W_1$, then so must $B$, and since $B$ is symmetric on a non-degenerate subspace and must be either symmetric or skew-symmetric, then $B$ is symmetric on $W$.  Thus $\epi(\chi) = 1$.
\end{proof}

We may also define the \emph{conformal symplectic group} $\CSp_{2n}$ to be the group of transformations which leave the symplectic form $\langle \cdot, \cdot \rangle$ on $V=\overline{\F}_q^{2n}$ defined in Section \ref{symplectic}, invariant up to a scalar multiple.  If we again write $\beta(g)$ for the scalar associated with $g \in \CSp_{2n}$, then $\beta$ is a multiplicative character.  With $F$ the standard Frobenius, we define
$$\CSp(2n, \F_q) = \CSp_{2n}^F,$$
and $\beta: \CSp(2n, \F_q) \rightarrow \F_q^{\times}$ is a homomorphism with kernel given by $\Sp(2n, \F_q)$.  Note that a skew-symplectic element $y \in \CSp(2n, \F_q)$ is one with the property $\beta(y) = -1$, and if $q$ is even this is just an element in $\Sp(2n, \F_q)$.  As in the orthogonal case, for any scalar $\alpha \in \F_q^{\times}$ we have $\beta(\alpha I) = \alpha^2$, and so when $q$ is even we again have a direct product $\CSp(2n, \F_q) = Z \times \Sp(2n, \F_q)$, where $Z \cong \F_q^{\times}$ is the group of scalar matrices.

\subsection{Summary of known twisted reality results}

We conclude by gathering known results, from this paper and others, on twisted reality results for finite classical groups.  The following strongly suggests that there should be some twisted geometric version of our main result Theorem \ref{SimpleTO}.

\begin{theorem} \label{TwistList}Let $q$ be any prime power.  Suppose $G$ is any of the following groups, with the specified automorphism $\iota$ of $G$ such that $\iota^2 = 1$:
\begin{enumerate}
\item $\GL(n,\F_q)$ or $\GU(n, \F_q)$, with $\iota(g) = {^\top g}^{-1}$,
\item $\gO^{\pm}(n, \F_q)$, $\SO^{\pm}(4m, \F_q)$, or $\SO(2n+1, \F_q)$ with $\iota(g) = g$,
\item $\CO^{\pm}(2n, \F_q)$ or $\CSO^{\pm}(4m, \F_q)$ with $\iota(g) = \beta(g)^{-1} g$,
\item $\SO^{\pm}(4m+2, \F_q)$  with $\iota = \sigma$,
\item $\CO^{\pm}(4m+2, \F_q)$ or $\CSO^{\pm}(4m+2, \F_q)$ with $\iota(g)= \beta(g)^{-1} \sigma(g)$,
\item $\Sp(2n, \F_q)$ with $\iota(g) = ygy^{-1}$ and $y$ a skew-symplectic involution, 
\item $\CSp(2n, \F_q)$ with $\iota(g) = \beta(g)^{-1} ygy^{-1}$ and $y$ a skew-symplectic involution,
\end{enumerate}
or the projective versions of any of these groups with the induced automorphism.  Then we have:
\begin{itemize}
\item[(a):] $\epi(\chi) = 1$ for all $\chi \in \Irr(G)$, and
\item[(b):] every $g \in G$ is strongly $\iota$-real.
\end{itemize}
\end{theorem}
\begin{proof}  For (a), the statement follows for $\GL(n, \F_q)$ for any $q$ by \cite[IV.6, Example 5]{Mac95}, and for $\GU(n, \F_q)$ by \cite[Corollary 5.2]{TV07}.  It follows for $\gO^{\pm}(2n, \F_q)$ and $\SO^{\pm}(2n, \F_q)$ for $q$ odd by \cite[Theorems 1 and 2]{Gow85} and \cite[Theorem 5.1(ii)]{TV17}, and for $q$ even by Theorems \ref{OisTO}, \ref{SO4misTO}, and \ref{SOIndicators1} above.  For $\SO(2n+1, \F_q)$ with $q$ odd, (a) follows also by \cite[Theorem 2]{Gow85}, and with $q$ even follows from \cite[Theorem 4.2]{V17}.  For $\CO^{\pm}(2n, \F_q)$ and $\CSO^{\pm}(2n, \F_q)$ with any $q$, the statement is given by Proposition \ref{CSOtwist} above.  For $\Sp(2n, \F_q)$ with $q$ odd, this is proved in Section \ref{symplectic} above, or is given by \cite[Theorem 6.1(ii)]{V05}, and for $\CSp(2n, \F_q)$ with $q$ odd this is \cite[Theorem 6.2]{V05}.  When $q$ is even, the statement for $\Sp(2n, \F_q) \cong \SO(2n+1, \F_q)$ is \cite[Theorem 4.2]{V17}, and for $\CSp(2n, \F_q) \cong \Sp(2n, \F_q) \times \F_q^{\times}$ this follows from an argument exactly like that in the third paragraph of the proof of Proposition \ref{CSOtwist}.  For the projective versions of all of these groups, statement (a) follows directly from Lemma \ref{modZ}.

For statement (b), the case for $G = \GL(n, \F_q)$ is the classical result that any invertible matrix is conjugated to its transpose by a symmetric matrix.  The cases of $\Sp(2n, \F_q)$ and $\gO^{\pm}(n, \F_q)$ with $q$ odd were proved originally by Wonenburger \cite{Wo66}.  Many of the other results are proved in various forms in several references, although these are largely generalized in \cite{RV18}, which covers all cases except for $\gO^{\pm}(2n, \F_q)$ and $\SO^{\pm}(2n, \F_q)$ with $q$ even, which are proved in \cite{ElNo82, Gow81} and in \cite{Ra11} and Proposition \ref{sigmareal}, respectively.  The fact that the projective versions of these groups are also strongly $\iota$-real follows immediately.
\end{proof}

We note that, included in the list of groups in Theorem \ref{TwistList}, are all classical simple groups of adjoint type over a finite field (that is, all finite groups of Lie type which are simple adjoint of type $A_n$, ${^2 A_n}$, $B_n$, $C_n$, $D_n$, or ${^2 D_n}$).

\end{document}